\documentclass[reqno]{amsart}

\usepackage{amsfonts}
\usepackage{amssymb}
\usepackage{amsmath}
\usepackage{mathrsfs}
\usepackage{amsthm}
\usepackage[colorlinks,linkcolor=blue,citecolor=blue]{hyperref}
\newtheorem{theorem}{Theorem}[section]

\newtheorem{lemma}[theorem]{Lemma}
\newtheorem{proposition}[theorem]{Proposition}
\newtheorem{question}[theorem]{Question}

\theoremstyle{remark}
\newtheorem{remark}[theorem]{Remark}

\theoremstyle{definition}
\newtheorem{definition}[theorem]{Definition}

\numberwithin{equation}{section}

\begin{document}

\title[Non-convexity of level sets]
{Non-convexity of level sets for solutions to $k$-Hessian equations in exterior domains}

\author[B. Wang]{Bo Wang}
\address[B. Wang]{School of Mathematics and Statistics\\
Beijing Institute of Technology\\
100081 Beijing\\
P.R. China}
\email{wangbo89630@bit.edu.cn} 

\author[C. Wang]{Cong Wang}
\address[C. Wang]{School of Statistics\\
University of International Business and Economics\\
100029 Beijing\\
P.R. China}
\email{cong\_wang@uibe.edu.cn}

\author[Z. Wang]{Zhizhang Wang}
\address[Z. Wang]{School of Mathematical Science\\
Fudan University\\
200433 Shanghai\\
P.R. China}
\email{zzwang@fudan.edu.cn}



\begin{abstract}
	In this paper, we provide examples to show that for $1 \leq k \leq n/2$, solutions to $k$-Hessian equations $S_k(D^2u)=1$ in the exterior of a strictly convex domain need not be quasiconvex, when prescribing quadratic growth at infinity. Additionally, we give a new proof for the quasiconvexity of harmonic functions in such exterior domains that decay to zero at infinity.
\end{abstract}

\maketitle 
  
\section{Introduction}

In this paper, we are concerned with the quasiconvexity of solutions to the exterior Dirichlet problem for $k$-Hessian equations
\begin{equation}\label{eq:EDP}
	\begin{cases}
		S_k(D^2u)=1  & \text{in }\mathbb{R}^n\setminus\bar\Omega_0, \\
		u=0  & \text{on }\partial\Omega_0.
	    \end{cases}
\end{equation}
The solution $u$ to \eqref{eq:EDP} is called quasiconvex if its sublevel sets are all convex (by extending $u=0$ in $\Omega_0$).
Here $\Omega_0$ is a bounded convex domain in $\mathbb{R}^n$, $n\geq3$, and the $k$-Hessian operator is defined by
$$S_k(D^2u):=S_k(\lambda)=\sum_{1\leq i_1<\cdots<i_k\leq n}\lambda_{i_1}\cdots\lambda_{i_k},$$
where $k=1, \cdots, n$, $\lambda=(\lambda_1,\cdots,\lambda_n)$ denotes eigenvalues of the Hessian matrix $D^2u$. Especially, if $k=1,n$, it becomes the Laplace operator and the Monge-Amp\`{e}re operator, respectively. 

To guarantee the well-posedness of problem \eqref{eq:EDP}, one needs to prescribe appropriate asymptotic behavior at infinity which is closely related to the Liouville-type theorems. For the Monge-Amp\`ere equation, i.e. $k=n$,
\begin{equation}\label{eq:MA}
	\det(D^2u)=1\quad\text{in }\mathbb{R}^n\setminus\bar{\Omega}_0,
\end{equation}
Caffarelli and Li \cite{Caffarelli-Li-2003} proved that any locally convex solution $u$ to \eqref{eq:MA} ($n\geq3$) exhibits the asymptotic behavior 
\begin{equation}\label{eq:asy}
	\lim_{|x|\to\infty}\bigg(u(x)-\bigg(\frac{1}{2}x^TAx+b\cdot x+c\bigg)\bigg)=0
\end{equation}
at infinity for some real symmetric positive definite matrix $A\in \mathbb{R}^{n\times n}$ with $\det(A)=1$, vector $b \in \mathbb{R}^n$ and constant $c \in \mathbb{R}$ such that. With such prescribed asymptotic behavior at infinity, they also proved the existence and uniqueness result for the exterior Dirichlet problem. In \cite{Bao-Li-Li}, Bao, Li and Li extended the existence and uniqueness result in \cite{Caffarelli-Li-2003} to $k$-Hessian equations with $2\leq k\leq n$. We also refer to \cite{Bao-Wang-2024, Li-Lu, Li-Xiao, Ma-Zhang,Wang-Wang} for more existence results. 

We now turn to the quasiconvexity of solutions to the problem \eqref{eq:EDP} satisfying the asymptotic behavior \eqref{eq:asy}. Geometrically, the quadratic growth \eqref{eq:asy} ensures that the sublevel sets of $u$ are asymptotically ellipsoidal, and thus convex, at infinity. Since the inner domain $\Omega_0$ is also convex, it is expected that this convexity is inherited globally by all sublevel sets. For the case $k=n$, the admissible solution to \eqref{eq:EDP} satisfying \eqref{eq:asy} is locally convex, and thus is quasiconvex. For general $1\leq k\leq n$, Wang-Bao \cite{Wang-Bao} considered the radial symmetric setting where $\Omega_0$ is a ball centered at the origin with radius $r$, $b=0$, and $A=a^*I$, where $a^*=(C_n^k)^{-1/k}\text{ and } C_n^k=\frac{n!}{k!(n-k)!}$.
They showed that the $k$-convex solution to \eqref{eq:EDP} satisfying \eqref{eq:asy} takes the explicit form
$$u(x)=a^*\int_r^{|x|}(s^k+\alpha s^{k-n})^{1/k}\mathrm{d}s,$$
where
$\alpha$ is a constant uniquely determined by constant $c$ in \eqref{eq:asy}. This solution is radially symmetric and not necessarily convex. However, the solution is increasing with respect to $|x|$, meaning that all sublevel sets of $u$ are concentric spheres and $u$ is therefore quasiconvex. 

Thus a natural question arises: 

\begin{question}\label{Q}
	For general $1\leq k\leq n-1$, do $k$-convex solutions to the problem \eqref{eq:EDP} satisfying the prescribed asymptotic behavior \eqref{eq:asy} inherit the ring-convexity of the exterior domain? That is, is every such solution quasiconvex?
\end{question}

In this paper, we provide a negative answer to this question by constructing explicit counterexamples of $k$-convex solutions that fail to be quasiconvex. Before stating our main theorems rigorously, we briefly review the related literatures.

Extensive studies have been devoted to finding sufficient conditions to ensure the convexity of level sets, particularly for linear and semi-linear equations.  

For exterior domains, a positive result has been established for harmonic functions. Specifically, if $\Omega_0$ is bounded convex in $\mathbb{R}^n$ $(n\geq3)$, then the harmonic function in $\mathbb{R}^n\setminus\bar{\Omega}_0$ with the condition
\begin{equation*}
u=1 \text{ on }\partial\Omega_0\quad\text{and}\quad\lim_{|x|\to\infty} u(x)=0
\end{equation*}
is quasiconcave (superlevel sets are all convex by extending $u=1$ in $\Omega_0$) due to Lewis \cite{Lewis-1977} through a macroscopic argument, and a simpler proof was given by Kawohl \cite{Kawohl-1985}. This problem is a classical topic in mathematical physics and potential theory.

For bounded convex rings (i.e., a domain $U\subset\mathbb{R}^n$ is called a convex ring if $$U=\Omega\setminus\bar{\Omega}_0,$$ 
where $\Omega_0$ and $\Omega$ are two bounded convex domains in $\mathbb{R}^n$ such that $\bar{\Omega}_0\subset\Omega$.), the foundationally positive result was established by Gabriel \cite{Gabriel-1957} for harmonic functions with boundary condition
\begin{equation}\label{eq:u-0-1}
	u=1\text{ on }\partial\Omega_0\quad\text{and}\quad u=0\text{ on }\partial\Omega.
\end{equation}
Lewis \cite{Lewis-1977} extended Gabriel's results to $p$-harmonic functions. After that, Ma, Ou and Zhang \cite{Ma-Ou-Zhang-CPAM} presented a new proof of Gabriel's and Lewis' results by combining Gaussian curvature estimates and deformation process. Caffarelli and Spruck \cite{Caffarelli-Spruck-1982} studied a type of semi-linear equations $\Delta u+f(u)=0$ with \eqref{eq:u-0-1} that are in connection with free boundary problems and proved the quasiconcavity of solutions. Later, Kawohl \cite{Kawohl-1985} proposed a strategy of using the maximum principle for quasiconcave envelope to study the convexity of level sets of solutions to semi-linear equations. This strategy has been successfully extended to establish positive results for a broad class of equations of the type 
\begin{equation}\label{eq:FNE}
	F(x,u,Du,D^2u)=0
\end{equation} 
that are possibly fully nonlinear, we refer to the works \cite{Bian-Long-Salani-2009, Colesanti-Salani-2003, Cuoghi-Salani-2006, Longinetti-Slani-2207} for more results. Contrast to the macroscopic works as above, we also would like to mention a microscopic direction in the study of the convexity, which seeks to establish constant rank theorems for the second fundamental forms of level sets and serve as an extremely useful tool for proving convexity. Pioneered by Korevaar \cite{Korevaar-1990} for $p$-harmonic functions, this line was extended by Xu \cite{Xu-CVPDE} to quasilinear equations. For the fully nonlinear case \eqref{eq:FNE}, Bian, Guan, Ma and Xu \cite{Bian-2011} established a constant rank theorem under specific structural conditions on $F$, which was further refined by Guan and Xu \cite{Guan-Xu-JRAM}.

Compared to many positive results that solutions of PDEs in convex rings that have non-convex sub- or super-level sets, the negative results remain rare. Monneau and Shahgholian \cite{Monneau-Shahgholian-2005} first constructed, for $n=2$, a solution to $\Delta u+f(u)=0$ in a convex ring with non-convex superlevel sets. This result was later extended by Hamel, Nadirashvili and Sire \cite{Hamel-N-S} to arbitrary dimensions under more general conditions on $f$. Very recently, the last named author and Xiao \cite{Wang-Xiao} made progress in finding counterexamples for fully nonlinear equations. Specifically, they considered the Dirichlet problem for $k$-Hessian equations defined on a convex ring ($1\leq k\leq n/2$)
\begin{equation*}
	\begin{cases}
		S_k(D^2u)=1 & \text{in }U:=\Omega\setminus\bar{\Omega}_0, \\
		u=-M & \text{on }\partial\Omega_0, \\
		u=0 & \text{on }\partial\Omega,
	\end{cases}
\end{equation*}
and constructed an example that exhibits non-convex sublevel sets, provided $M$ is a sufficiently large constant.

For the reader's convenience, we give the following definitions of $k$-convex functions and quasiconvex functions. For a domain $\Omega$ in $\mathbb{R}^n$, we say that a function \( u \in C^2(\Omega) \) is \( k \)-convex, if \( \lambda(D^2u) \in \Gamma_k \) in \( \Omega \), where \( \Gamma_k \) is the  Garding's cone
\[
\Gamma_k = \{ \lambda \in \mathbb{R}^n: S_j(\lambda) > 0, \, \forall\, j = 1, \dots, k \}.
\]

\begin{definition}
	For any $k$-convex solution $u\in C^2(U)\cap C^0(\bar{U})$ to \begin{equation*}
	\begin{cases}
		S_k(D^2u)=1 & \text{in }U=\Omega\setminus\bar{\Omega}_0, \\
		u=0 & \text{on }\partial\Omega_0, \\
	\end{cases}
\end{equation*} 
we define
	\begin{equation*}
		\tilde{u}(x)=
		\begin{cases} 
			u(x), & x\in\bar{U}, \\
			0, & x\in \Omega_0.
		\end{cases}
	\end{equation*} 
	We say that $u$ is quasiconvex in $U$ if the sublevel sets of $\tilde{u}$
	$$\Sigma_t(\tilde{u}):=\{x\in\bar{\Omega}: \tilde{u}(x)\leq t\}$$
	are convex for all $t\in\mathbb{R}$.
\end{definition}

Denote 
$$\mathcal{A}_k=\{A: A \text{ is a real }n\times n \text{ symmetric positive definite matrix with }S_k(A)=1\}.$$  
Our main result is the following, which extends the main results of \cite{Wang-Xiao} to exterior domains.
\begin{theorem}\label{thm:main}
	Let $n\geq3$  and $2\leq k \le n/2$. 
	For any given $A\in\mathcal{A}_k$, $b\in\mathbb{R}^n$, 
	there exist a smooth strictly convex domain $\Omega_0$ and a constant $c^*$, 
	depending only on $n$, $k$, $A$ and $b$, 
	such that for every $c\geq c^*$, the problem
	\begin{equation}\label{eq:EDP-1}
		\begin{cases}
			S_k(D^2u)=1 \quad \text{in }\mathbb{R}^n\setminus\bar\Omega_0, \\
			u=0 \quad \text{on }\partial\Omega_0, \\
			\lim\limits_{|x|\to\infty}\left(u(x)-\left(\frac{1}{2}x^TAx+b\cdot x+c\right)\right)=0
	    \end{cases}
	\end{equation}
	admits a unique $k$-convex solution $u\in C^\infty(\mathbb{R}^n\setminus\Omega_0)$ which is not quasiconvex.
\end{theorem}  

For the Laplace case, we have the similar result under an additional condition on  $A$.

\begin{theorem}\label{thm:main-1}
	Let $n\geq3$ and $k=1$. For any $A\in\mathcal{A}_1$, denote $\lambda_{\max}(A)$ the maximum eigenvalue of $A$. If $\lambda_{\max}(A)<1/2$, $b\in\mathbb{R}^n$, 
	there exist a smooth strictly convex domain $\Omega_0$ and a constant $c^*$, 
	depending only on $n$, $A$ and $b$, 
	such that for every $c\geq c^*$, the problem \eqref{eq:EDP-1} admits a unique solution $u\in C^\infty(\mathbb{R}^n\setminus\Omega_0)$ which is not quasiconvex.
\end{theorem}

\begin{remark}\label{rmk}
	In Theorems \ref{thm:main} and \ref{thm:main-1}, there actually exist infinitely many domains $\Omega_0$ that satisfy the above conclusions.
\end{remark}

In the last part of the paper, we revisit the classical result in \cite{Lewis-1977} and \cite{Kawohl-1985} for the Laplace equation in exterior domains. We provide a new proof here in the microscopic perspective by using the explicit decay estimate of the Gaussian curvature for level sets at infinity, combining with the superharmonicity of curvature functions established by Ma-Zhang \cite{Ma-Zhang-2021}.   

\begin{theorem}[Microscopic version proof of \cite{Lewis-1977} and \cite{Kawohl-1985}]\label{thm:main-H}
	Let $n\geq3$ and $\Omega_1$ be a smooth, strictly convex domain in $\mathbb{R}^n$. Then the solution $u\in C^\infty(\mathbb{R}^n\setminus\Omega_1)$ to the problem
	\begin{equation}\label{eq:H-EDP}
		\begin{cases}
			\Delta u=0 \,\qquad \text{in }\mathbb{R}^n\setminus\bar{\Omega}_1, \\
			u=1 \quad\quad\quad \text{on }\partial\Omega_1, \\
			\lim\limits_{|x|\to\infty} u(x)=0,
		\end{cases}	
	\end{equation}
	is quasiconcave. Moreover, superlevel sets of $u$ are all strictly convex by extending $u = 1$ in $\Omega_1$.
\end{theorem}

Now we sketch the proof of Theorems \ref{thm:main} and \ref{thm:main-1}. By an normalization, we can reduce Theorems \ref{thm:main} and \ref{thm:main-1} to the special and simple case where $A\in\mathcal{A}_k$ is diagonal and $b$ vanishes. We choose the candidate domain $\Omega_0=B_\varepsilon(x_0)$ to be a ball centered at $x_0$ located away from the origin and a sufficiently small radius $\varepsilon$. We start by studying the approximating problem in the bounded convex ring $\Omega_R^\varepsilon$ (see Section \ref{sec:u-e-R} for the precise definition). The existence of solutions $u^{\varepsilon,R}$ follows from Guan \cite{Guan-Duke}, provided that we can construct a suitable subsolution by utilizing generalized symmetric functions. To investigate the limit of $u^{\varepsilon,R}$ as $R\to\infty$ and $\varepsilon\to 0$, we establish uniform $C^2$ estimates for $u^{\varepsilon,R}$. The singularity at $x_0$ is addressed in the spirit of \cite{Wang-Xiao} by employing Hessian measures theory developed in \cite{Trudinger-Wang-II}.

The results in Theorems \ref{thm:main} and \ref{thm:main-1} illustrate that the quasiconvexity of solutions is highly sensitive to minor deformations of the domain. To the best of our knowledge, the results are new especially for fully nonlinear equations in exterior domains.

The organization of the paper is as follows. In Section \ref{sec:u-e-R}, we solve the approximating Dirichlet problem
for $k$-Hessian equations in a non-concentric convex ring and derive the uniform $C^2$ estimates for approximating solutions. In Section \ref{sec:u-e}, we investigate the limit as the outer boundary expanding to infinity and the inner boundary shrinking; the latter employs the theory of $k$-Hessian measures. Theorems \ref {thm:main} and \ref{thm:main-1} will be proved in Section \ref{sec:proof}. In Section \ref{sec:convexity-H}, we analyze the asymptotic behavior of Gaussian curvatures functions at infinity and carry out a deformation process to prove Theorem \ref{thm:main-H}.  

Throughout this paper, denote $\lambda_{\text{max}}(A)=\max\limits_{1\leq i\leq n}\lambda_i(A)$ and $\lambda_{\text{min}}(A)=\min\limits_{1\leq i\leq n}\lambda_i(A)$ for $A\in\mathcal{A}_k$. We denote $S_k^{ij}(D^2u)=\frac{\partial S_k(D^2u)}{\partial u_{ij}}$ and will also write $S_k^{ij}(D^2u)$ as $S_k^{ij}$ for simplicity when there is no ambiguity.

\section{Approximating problems in convex rings}\label{sec:u-e-R}

In this section, we investigate the Dirichlet problem of $k$-Hessian equations in the bounded non-concentric convex rings. To formulate the problem, we assume throughout this and the next section that 
$$A=\text{diag}(a_1,\cdots,a_n)\in\mathcal{A}_k,$$
where $n\geq3$ and $1\leq k\leq n$. When $k=1$, we further assume $\lambda_{\mathrm{max}}(A)<\frac{1}{2}$.
Let 
$$s=\frac{1}{2}x^TAx=\frac{1}{2}\sum_{i=1}^n a_ix_i^2\quad\text{and}\quad E_R(\bar{x})=\bigg\{x\in\mathbb{R}^n: \frac{1}{2}(x-\bar{x})^TA(x-\bar{x})<R\bigg\}.$$
Denote $\Omega_R^\varepsilon=E_R(0)\setminus\bar{B}_\varepsilon(x_0)$, where
$x_0\in B_{\frac{1}{2}}(0)\setminus\{0\}$, $0<\varepsilon<\frac{1}{2}$, $R>3R_0$ with $R_0\gg\lambda_{\text{max}}(A)$.
We consider the following approximating problem
\begin{equation}\label{eq:u-e-R}
	\begin{cases}
		S_k(D^2u)=1 & \text{in }\Omega_R^\varepsilon, \\
		u=0 & \text{on }\partial B_\varepsilon(x_0), \\
		u=\underline{u} & \text{on }\partial E_R(0).
	\end{cases}
\end{equation} 
In this formulation, $\underline{u}$ is a generalized symmetric function, as defined in \cite{Bao-Li-Li},
\begin{equation*}  
		\underline{u}(x):=\int_{R_0}^{s}\bigg(1+\alpha t^{-\frac{k}{2h_k(a)}}\bigg)^\frac{1}{k}\mathrm{d}t,\quad x\in\mathbb{R}^n\setminus\{0\},
\end{equation*}
where $\alpha$ is a positive parameter, $a=(a_1,\cdots,a_n)$ and $h_k(a)=\max\limits_{1\leq i\leq n}S_{k-1}(a)|_{a_i=0}\cdot a_i$ satisfies
\begin{equation}\label{eq:hk}
	1<\frac{k}{2h_k(a)}\leq \frac{n}{2}.
\end{equation}
Such exterior boundary data is to ensure the solution $u^{\varepsilon, R}$ to problem \eqref{eq:u-e-R} approximates the expected quadratic asymptotic behavior at infinity.  Our goal in what follows is to establish the solvability and uniform estimates for problem \eqref{eq:u-e-R}.

\subsection{Solvability of approximating problems.}\label{subsec:solvability}
By Proposition 2.1 in \cite{Bao-Li-Li}, $\underline{u}\in C^{\infty}(\mathbb{R}^n\setminus\{0\})$ is $k$-convex and satisfies
\begin{equation}\label{eq:sub-0}
	S_k(D^2\underline{u})\geq1\quad\text{in }\mathbb{R}^n\setminus\{0\}.
\end{equation}
 Moreover, \eqref{eq:hk} yields
\begin{equation}\label{eq:subsol-u-2}
	\begin{split}
	\underline{u} (x) & =s+\int_{R_0}^{s}\bigg(\bigg(1+\alpha t^{-\frac{k}{2h_k(a)}}\bigg)^\frac{1}{k}-1\bigg)\mathrm{d}t-R_0\\
	& = s+\mu(\alpha)-\int_s^\infty\bigg(\bigg(1+\alpha t^{-\frac{k}{2h_k(a)}}\bigg)^\frac{1}{k}-1\bigg)\mathrm{d}t \\
	& = \frac{1}{2}x^TAx + \mu(\alpha) + o(1)\quad\text{as }|x|\to\infty,
	\end{split}
\end{equation}  
where $\mu(\alpha)$ is a strictly increasing function with respect to $\alpha$ defined as
$$\mu(\alpha):= \int_{R_0}^{\infty} \bigg( \bigg( 1 + \alpha t^{-\frac{k}{2h_k(\alpha)}} \bigg)^{\frac{1}{k}} - 1 \bigg) \mathrm{d}t-R_0.$$
We first show the existence and uniqueness of the solution to problem \eqref{eq:u-e-R} by constructing appropriate subsolutions.

\begin{proposition}\label{prop:u-e-R-exist}
	There exists a positive constant $\alpha_0$, depending only on $n$, $k$, $\lambda_{\mathrm{max}}(A)$ and $R_0$, such that for every $\alpha\geq\alpha_0$, the problem \eqref{eq:u-e-R} has a unique $k$-convex solution $u^{\varepsilon, R} \in C^{\infty}(\bar{\Omega}_{R}^{\varepsilon})$. 
\end{proposition}
\begin{proof}
	Let
	\begin{equation*}
		\underline{u}^{\varepsilon}(x)=a^*(|x-x_0|^2-\varepsilon^2),
	\end{equation*}
	where the constant $a^*$ satisfies $S_k(a^*I)=1$, and is given by
	$$a^*=(C_n^k)^{-\frac{1}{k}}\quad\text{and}\quad C_n^k=\frac{n!}{k!(n-k)!}.$$ 
	We claim that there exists a $k$-convex function $\underline{u}^{\varepsilon,R}\in C^\infty(\mathbb{R}^n\setminus B_\varepsilon(x_0 ))$ satisfying 
	\begin{equation}\label{eq:subsol}
		\begin{cases}
			S_k(D^2\underline{u}^{\varepsilon,R})\geq 1 & \text{in }\mathbb{R}^n\setminus B_\varepsilon(x_0), \\
			\underline{u}^{\varepsilon,R}=\underline{u}^{\varepsilon} & \text{in }B_2(0)\setminus B_\varepsilon(x_0), \\
			\underline{u}^{\varepsilon,R}=\underline{u} & \text{in }\mathbb{R}^n\setminus E_{2R_0}(0),
		\end{cases}
	\end{equation}
	and
	\begin{equation*}
		\underline{u}^{\varepsilon,R}\geq\max\{\underline{u}^\varepsilon, \underline u\}\quad\text{in }E_{2R_0}(0)\setminus B_2(0),
	\end{equation*}
	provided that $\alpha$ is suitably large. Indeed, by a direct calculation and \eqref{eq:sub-0}, $\underline{u}^{\varepsilon}$ and $\underline{u}$ are smooth $k$-convex subsolutions to the equation $S_k(D^2u)=1$ in $\mathbb{R}^n\setminus \{0\}$. Moreover,
$$\underline{u}^{\varepsilon}-\underline{u}\geq-\underline{u}(x)\geq\int_{2\lambda_{\text{max}}(A)}^{R_0}\bigg(1+\alpha t^{-\frac{k}{2h_k(a)}}\bigg)^\frac{1}{k}\mathrm{d}t\geq1\quad\text{in }B_2(0)\setminus B_1(0),$$
and
$$\underline{u}-\underline{u}^{\varepsilon} \geq \int_{R_0}^{2R_0}\bigg(1+\alpha t^{-\frac{k}{2h_k(a)}}\bigg)^\frac{1}{k}\mathrm{d}t-\max_{\bar{E}_{3R_0}(0)}\underline{u}^{\varepsilon  }\geq1\quad\text{in }E_{3R_0}(0)\setminus E_{2R_0}(0),$$
for $\alpha\geq\alpha_0$ where $\alpha_0>$ is a large constant depending only on $n$, $k$, $\lambda_{\text{max}}(A)$ and $R_0$, and satisfies $\mu({\alpha_0})\geq0$. With the aid of Lemma 3.2 in \cite{P.Guan-2002} $(\delta=1)$ (see also Lemma 2.8 in \cite{Ma-Zhang}), there exists a $k$-convex function $\underline{u}^{\varepsilon,R}\in C^\infty(E_{3R_0}(0)\setminus B_1(0))$ satisfying $\underline{u}^{\varepsilon,R}\geq\max\{\underline{u}^\varepsilon, \underline u\}$ in $E_{3R_0}(0)\setminus B_1(0)$ and 
\begin{equation*}
		\begin{cases}
			S_k(D^2\underline{u}^{\varepsilon,R})\geq 1 & \text{in }E_{3R_0}(0)\setminus B_1(0), \\
			\underline{u}^{\varepsilon,R}=\underline{u}^{\varepsilon} & \text{in }\{\underline{u}^{\varepsilon}
			-\underline{u}>1\}\supset B_2(0)\setminus B_1(0), \\
			\underline{u}^{\varepsilon,R}=\underline{u} & \text{in }\{\underline{u}-\underline{u}^{\varepsilon}>1\}\supset E_{3R_0}(0)\setminus E_{2R_0}(0),
		\end{cases}
\end{equation*}
where the first line is due to the concavity of $\log S_k$, precisely, for some $|t(x)|\leq1$,
\[
\begin{split}
	\log S_k(D^2\underline{u}^{\varepsilon,R})(x) & \geq\log S_k\bigg(\frac{1+t(x)}{2}D^2\underline{u}^{\varepsilon}+\frac{1-t(x)}{2}D^2\underline{u}\bigg) \\
	& \geq\frac{1+t(x)}{2}\log S_k(D^2\underline{u}^{\varepsilon})+\frac{1-t(x)}{2}\log S_k(D^2\underline{u}) \\
	& \geq0\quad\quad\text{for }x\in\{|\underline{u}^{\varepsilon}
	-\underline{u}|<1\}.
\end{split}
\]  
The claim is thus proved by setting 
$$\underline{u}^{\varepsilon,R}=\underline{u}^\varepsilon\text{ in }B_1(0)\setminus B_\varepsilon(x_0)\quad\text{and}\quad \underline{u}^{\varepsilon,R}=\underline{u}\text{ in }\mathbb{R}^n\setminus E_{3R_0}(0).$$

By virtue of \eqref{eq:subsol}, $\underline{u}^{\varepsilon,R}$ is a smooth $k$-convex  subsolution to \eqref{eq:u-e-R} in $\Omega_R^\varepsilon$ and satisfies the boundary data on $\partial \Omega_R^\varepsilon$. The existence of the solution $u^{\varepsilon,R}$ is guaranteed by Guan \cite{Guan-Duke} and its uniqueness follows from the comparison principle.
\end{proof}

In the remainder of this section, we will always assume $\alpha\geq\alpha_0$ with $\alpha_0$ given by Proposition \ref{prop:u-e-R-exist}.

\subsection{$C^0$ estimates of $u^{\varepsilon,R}$}\label{subsec:C0-est}
We first establish uniform $C^0$ bound for the solution $u^{\varepsilon,R}$ to \eqref{eq:u-e-R} by utilizing explicit sub- and super- solutions. 
Let for $x\in\mathbb{R}^n$,
\begin{equation}\label{eq:psi}
	\psi(x)=\frac{1}{2}x^TAx+\mu(\alpha).
\end{equation}
Notice that the expansion \eqref{eq:subsol-u-2} implies
$$\underline{u}(x)\leq\frac{1}{2}x^TAx+\mu(\alpha).$$
Clearly, the supersolution $\psi\in C^\infty(\bar{\Omega}_R^\varepsilon)$ satisfies that
\begin{equation}\label{eq:supersol}
	\begin{cases}
		S_k(D^2\psi)=1 & \text{in }\Omega_R^\varepsilon, \\
		\psi\geq0 & \text{on }\partial B_\varepsilon(x_0), \\
		\psi\geq u^{\varepsilon,R} & \text{on }\partial E_R(0).
	\end{cases}
\end{equation}
Let $\underline{u}^{\varepsilon,R}$ be the subsolution found in the proof of Proposition \ref{prop:u-e-R-exist}. By \eqref{eq:subsol} and \eqref{eq:supersol}, the comparison principle yields the following $C^0$ estimates.

\begin{lemma}\label{lem:C0-est}
	Let $u^{\varepsilon,R}\in C^\infty(\bar{\Omega}_R^\varepsilon)$ be the $k$-convex solution to \eqref{eq:u-e-R}. Then
	$$\underline{u}^{\varepsilon,R}\leq u^{\varepsilon,R}\leq \psi\quad\text{in }\Omega_R^\varepsilon.$$
\end{lemma}

\subsection{$C^1$ estimates of $u^{\varepsilon,R}$}
By constructing barrier functions, we establish the following estimates for $|Du^{\varepsilon,R}|$ on the boundary $\partial\Omega_R^\varepsilon$, with precise dependence on $\varepsilon$ and $R$.

\begin{lemma}\label{lem:gra-bdy}
	Let $u^{\varepsilon,R}\in C^\infty(\bar{\Omega}_R^\varepsilon)$ be the $k$-convex solution to \eqref{eq:u-e-R}. Then 
	we have the estimates
	$$2a^*\varepsilon\leq|Du^{\varepsilon,R}|\leq C_\varepsilon\quad\text{on }\partial B_\varepsilon(x_0),$$
	and 
	$$
	\frac{1}{2}\sqrt{\lambda_{\mathrm{min}}(A)R}\leq|Du^{\varepsilon,R}|\leq C_R\quad\text{on }\partial E_R(0).$$
	Precisely, 
	\begin{equation*}
		C_\varepsilon = 
		\begin{cases}
			C\left( \frac{n}{k} - 2 \right) \varepsilon^{-1}, & \text{if } k < \frac{n}{2}, \\[10pt]
			C\varepsilon^{-1} |\log \varepsilon|^{-1}, & \text{if } k = \frac{n}{2}, \\[10pt]
			C\left( 2 - \frac{n}{k} \right) \varepsilon^{1-\frac{n}{k}}, & \text{if } k > \frac{n}{2},
		\end{cases}
	\end{equation*}
	and $C_R=C\sqrt{R}$, where $C$ is a  positive constant depending only on $n$, $k$, $\lambda_{\mathrm{max}}(A)$, $\lambda_{\mathrm{min}}(A)$, $R_0$ and $\alpha$ while independent of $\varepsilon$ and $R$. 
\end{lemma}
\begin{proof}
	Denote $\nu$ the unit inward normal of $\Omega_R^\varepsilon$, i.e., $\nu$ points into $E_R(0)\setminus\bar{B}_\varepsilon(x_0)$. Since $u^{\varepsilon,R} = \underline{u}^{\varepsilon,R}$ on $\partial B_\varepsilon(x_0)$ and $u^{\varepsilon,R} \geq \underline{u}^{\varepsilon,R}$ in $\Omega_R^\varepsilon$,
	combining \eqref{eq:subsol}, we have
    \[
        \frac{\partial u^{\varepsilon,R}}{\partial \nu} \geq \frac{\partial \underline{u}^{\varepsilon,R}}{\partial \nu}=\frac{\partial \underline{u}^{\varepsilon}}{\partial \nu}=2a^*\varepsilon \quad \text{on } \partial B_\varepsilon(x_0).
    \]
	On the other hand, we define
    \begin{equation*}
        \bar{u}^\varepsilon(x) = 
        \begin{cases}
            C \bigg( 1 - \bigg( \dfrac{|x-x_0|}{\varepsilon} \bigg)^{2 - \frac{n}{k}} \bigg), & \text{if } k < \frac{n}{2}, \\[10pt]
            C \bigg( 1-\dfrac{\log |x-x_0|}{\log \varepsilon} \bigg), & \text{if } k = \frac{n}{2}, \\[10pt]
            C \bigg(|x-x_0|^{2-\frac{n}{k}}-\varepsilon^{2-\frac{n}{k}}\bigg), & \text{if } k > \frac{n}{2}.
        \end{cases}
    \end{equation*}
    It is straightforward to verify that in all three cases, $\bar{u}^\varepsilon$ are $k$-convex and satisfies $S_k(D^2 \bar{u}^\varepsilon) = 0$ and $\bar{u}^\varepsilon=0$ on $\partial B_\varepsilon(x_0)$. Take $C=C(n,k,\lambda_{\text{max}}(A),R_0,\alpha)$ large enough such that $\bar{u}^\varepsilon\geq2\lambda_{\text{max}}(A)+\mu(\alpha)\geq\psi\geq u^{\varepsilon,R}$ on $\partial B_2(0)$. The comparison principle implies $u^{\varepsilon,R}\leq\bar{u}^\varepsilon$ in $B_2(0)\setminus\bar{B}_\varepsilon(x_0)$. A direct calculation yields
	$$\frac{\partial u^{\varepsilon,R}}{\partial \nu}\leq\frac{\partial\bar{u}^\varepsilon}{\partial\nu} = C_\varepsilon \quad\text{on }\partial B_\varepsilon(x_0).$$
    From these, we obtain the desired gradient estimate on $\partial B_\varepsilon(x_0)$.

	Since $u^{\varepsilon,R}=\underline{u}^{\varepsilon,R}$ on $\partial E_R(0)$ and $u^{\varepsilon,R}\geq\underline{u}^{\varepsilon,R}$ in $\Omega_R^\varepsilon$, again by \eqref{eq:subsol}, we have
	\[
        \frac{\partial u^{\varepsilon,R}}{\partial \nu} \geq \frac{\partial \underline{u}^{\varepsilon,R}}{\partial \nu}=\frac{\partial \underline{u}}{\partial \nu}\geq-\bigg(1+\alpha R^{-\frac{k}{2h_k(a)}}\bigg)^{\frac{1}{k}}\sqrt{2\lambda_{\text{max}}(A)R}\geq-C\sqrt{R}   \quad \text{on } \partial E_R(0).
    \]
	Let 
	\[\bar{u}^R(x)=\frac{1}{2}\bigg(\frac{1}{2}x^TAx-R\bigg)+ \underline{u}|_{\partial E_R(0)}.\] 
	Then $\bar{u}^R$ is $k$-convex and satisfies $S_k(D^2\bar{u}^R)<1$, with $\bar{u}^R\geq0$ on $\partial B_\varepsilon(x_0)$ and $\bar{u}^R=u^{\varepsilon,R}$ on $\partial E_R(0)$. By the comparison principle, we get
	\[ 
	\frac{\partial u^{\varepsilon,R}}{\partial \nu}\leq\frac{\partial\bar{u}^R}{\partial\nu}\leq-\frac{1}{2}\sqrt{\lambda_{\text{min}}(A)R} \quad \text{on } \partial E_R (0).
	\]
	Hence, we obtain the desired   gradient estimate on $\partial E_R(0)$.
\end{proof}

A standard argument illustrates that the global gradient estimate can be inferred from the boundary gradient estimate.

\begin{lemma}\label{lem:C1_est}
    Let $u^{\varepsilon,R} \in C^\infty(\bar{\Omega}_R^\varepsilon)$ be the $k$-convex solution to \eqref{eq:u-e-R}. Then
	$$|Du^{\varepsilon,R}|\leq\max\{ C_\varepsilon, C_R\}\quad\text{in }\Omega_R^\varepsilon,$$
	where $C_\varepsilon$ and $C_R$ are given by Lemma \ref{lem:gra-bdy}.
\end{lemma}
\begin{proof}
	Differentiating the equation in \eqref{eq:u-e-R} with respect to $x_l$ gives
	$$\sum_{i,j}S_k^{ij} u^{\varepsilon,R}_{ijl} = 0.$$
	By a direct calculation and the ellipticity of $S_k^{ij}$, we have
	\begin{equation*}
		\begin{split}
			\sum_{i,j}S_k^{ij} (|Du^{\varepsilon,R}|^2)_{ij} & = 2\sum_{i,j}S_k^{ij} u^{\varepsilon,R}_l u^{\varepsilon,R}_{lij} + 2\sum_{i,j}S_k^{ij} u^{\varepsilon,R}_{li} u^{\varepsilon,R}_{lj} \\
			& = 2\sum_{i,j}S_k^{ij} u^{\varepsilon,R}_{li} u^{\varepsilon,R}_{lj} \geq 0.
		\end{split}
	\end{equation*}
	By virtue of the maximum principle, 
    \begin{equation}\label{eq:mp}
        \max_{\bar{\Omega}_R^\varepsilon} |D u^{\varepsilon,R}| = \max_{\partial\Omega_R^\varepsilon} |D u^{\varepsilon,R}|.
    \end{equation}
	The conclusion follows from Lemma \ref{lem:gra-bdy}.
\end{proof}
 
We proceed to establish the gradient estimates near and away the boundary $\partial\Omega_R^\varepsilon$ that will be used in the next subsection and section.
\begin{lemma}\label{lem:near-bdy}
Let $u^{\varepsilon,R} \in C^\infty(\bar{\Omega}_R^\varepsilon)$ be the $k$-convex solution to \eqref{eq:u-e-R}. Then we have the estimates 
\begin{alignat}{2}
	|Du^{\varepsilon, R}| &\le C_\varepsilon \quad && \text{in } B_2(0) \setminus B_\varepsilon(x_0), \label{eq:C1-1}\\
	|Du^{\varepsilon, R}| &\le C_R           \quad && \text{in } E_R(0) \setminus E_{\frac{2}{3}R}(0), \label{eq:C1-2}
\end{alignat}
	and for given small constant $\delta>0$ and any $0<\varepsilon<\delta$,
\begin{alignat}{2}
	|Du^{\varepsilon, R}| &\le C             \quad && \text{in } B_{2}(0) \setminus B_\delta(x_0), \label{eq:C1-3}
\end{alignat}
where $C_\varepsilon$ and $C_R$ are given by Lemma \ref{lem:gra-bdy}, and $C$ is a positive constant depending only on $n$, $k$, $\lambda_{\mathrm{max}}(A)$, $\lambda_{\mathrm{min}}(A)$, $\alpha$ and $\delta$ while independent of $\varepsilon$ and $R$.
\end{lemma}
\begin{proof}
	Analogously to \eqref{eq:mp}, we have
    \begin{equation*}
    \begin{aligned}
        \sup_{B_2(0) \setminus B_\varepsilon(x_0)} |Du^{\varepsilon,R}| 
        &= \max \Big\{ \sup_{\partial B_\varepsilon(x_0)} |Du^{\varepsilon,R}|, \sup_{\partial B_2(0)} |Du^{\varepsilon,R}| \Big\}.
    \end{aligned}
    \end{equation*}
	For the term on $\partial B_{2}(0)$, applying the interior gradient estimate in \cite{Chou-Wang}, we obtain
    \begin{equation}\label{eq:gra-pB2}
        \sup_{\partial B_2(0)} |Du^{\varepsilon,R}| \leq C\operatorname*{osc}_{B_3(0) \setminus B_1(0)} u^{\varepsilon,R} \leq C \sup_{B_3(0) \setminus B_1(0)} \psi \leq C.
    \end{equation}
    Combining with Lemma \ref{lem:gra-bdy}, we obtain \eqref{eq:C1-1}. By an argument similar to \eqref{eq:gra-pB2}, we get $\sup\limits_{\partial B_\delta(x_0)}|Du^{\varepsilon,R}|\leq C$ and thus \eqref{eq:C1-3} follows.

	Also, we have
	\begin{equation*} 
    \begin{aligned}
        \sup_{E_R(0) \setminus E_{\frac{2}{3}R}(0)} |Du^{\varepsilon,R}| 
        &= \max \Big\{ \sup_{\partial E_R(0)} |Du^{\varepsilon,R}|, \sup_{\partial E_{\frac{2}{3}R}(0)} |Du^{\varepsilon,R}| \Big\}.
    \end{aligned}
    \end{equation*}
    On $\partial E_{\frac{2}{3}R}(0)$, the interior gradient estimate in \cite{Chou-Wang} implies
    \begin{equation*}
        \sup_{\partial E_{\frac{2}{3}R}(0)} |Du^{\varepsilon,R}| \leq \frac{C}{\sqrt{R}} \operatorname*{osc}_{E_R(0) \setminus E_{\frac{R}{3}}(0)} u^{\varepsilon,R} \leq \frac{C}{\sqrt R} \sup_{E_R(0) \setminus E_{\frac{R}{3}}(0)}\psi \leq C\sqrt{R}.
    \end{equation*}
    Again combining with Lemma \ref{lem:gra-bdy}, we obtain \eqref{eq:C1-2}.
\end{proof}
\begin{remark}\label{rmk:Du-est}
	$B_2(0)$ in \eqref{eq:C1-1} and \eqref{eq:C1-3} can be replaced by any given large ball centered at the origin.
\end{remark} 

\subsection{$C^2$ estimates of $u^{\varepsilon,R}$} We derive the following estimates for $|D^2u^{\varepsilon,R}|$ on the boundary $\partial\Omega_R^\varepsilon$, by sequentially bounding the double tangential, tangential-normal, and double normal derivatives, with precise dependence on $\varepsilon$ and $R$.
\begin{lemma}\label{lem:C2-est}
	Let $u^{\varepsilon,R}\in C^\infty(\bar{\Omega}_R^\varepsilon)$ be the $k$-convex solution to \eqref{eq:u-e-R}. Then we have the estimates 
\begin{equation*}
	|D^2u^{\varepsilon,R}|\leq C\cdot C_\varepsilon \varepsilon^{-1}\quad\text{on }\partial B_\varepsilon(x_0),
\end{equation*}
and 
\begin{equation*}
	|D^2u^{\varepsilon,R}|\leq C\quad\text{on }\partial E _R(0).
\end{equation*}
	where $C$ is a positive constant depending only on $n$, $k$, $\lambda_{\mathrm{max}}(A)$, $\lambda_{\mathrm{min}}(A)$, $R_0$ and $\alpha$ while independent of $\varepsilon$ and $R$, and $C_\varepsilon$ is as in Lemma \ref{lem:gra-bdy}.
\end{lemma}
\begin{proof} 
	In the proof, we drop the superscripts $\varepsilon$ and $R$ and write $u^{\varepsilon,R}$ as $u$ for simplicity.

\textit{Step 1. Double tangential derivatives on $\partial\Omega_R^\varepsilon$.} For any fixed point $p\in\partial\Omega_R^\varepsilon$, we choose the local coordinate $\{\tilde{x}_1,\cdots,\tilde{x}_n\}$ in a neighborhood of $p$ such that $p$ is the origin and the positive $\tilde{x}_n$-axis is the inner normal of $\partial\Omega_R^\varepsilon$ (pointing into $\Omega_R^\varepsilon$). Then
$\partial  \Omega_R^\varepsilon$ can be locally represented by $\tilde{x}_n=\rho(\tilde{x}')$, $\rho(\tilde{0}')=0$ and $D'\rho(\tilde{0}')=0'$, where $\tilde{x}'=(\tilde x_1,\cdots,\tilde x_{n-1})$. Differentiating the boundary data $u(\tilde{x}',\rho(\tilde{x}'))\equiv\text{const}$ once and twice yields for $1\leq \alpha,\ \beta\leq n-1$,
\begin{equation}\label{eq:tan-d1}
0 = u_{\alpha}(\tilde{x}', \rho(\tilde{x}')) + u_n(\tilde{x}', \rho(\tilde{x}')) \rho_\alpha(\tilde{x}'),
\end{equation}
and
\[
\begin{aligned}
0 &= u_{\alpha\beta}(\tilde 0) + u_{\alpha n}(\tilde 0) \rho_\beta(\tilde 0') + u_{\beta n}(\tilde 0) \rho_\alpha(\tilde 0') + u_{nn}(\tilde 0) \rho_\alpha(\tilde 0') \rho_\beta(\tilde{0}') + u_n(\tilde 0) \rho_{\alpha\beta}(\tilde{0}') \\
&= u_{\alpha\beta}(\tilde 0) + u_n(\tilde 0)\rho_{\alpha\beta}(\tilde 0').
\end{aligned}
\]
Hence,  
\begin{equation}\label{eq:TT-I}
	u_{\alpha\beta}(\tilde 0) = -u_n(\tilde 0) \rho_{\alpha\beta}(\tilde 0').
\end{equation}
For the boundary \(\partial B_\varepsilon(x_0)\), we have \(\rho(\tilde{x}') = \sqrt{\varepsilon^2 - |\tilde{x}'|^2}-\varepsilon \) and
\[
\rho_{\alpha\beta}(\tilde{x}') = -\frac{\delta_{\alpha\beta}}{\sqrt{\varepsilon^2 - |\tilde{x}'|^2}} -\frac{\tilde x_\alpha \tilde x_\beta}{(\varepsilon^2 - |\tilde{x}'|^2)^{\frac{3}{2}}},
\]
and so 
\begin{equation}\label{eq:rho-11}
	\rho_{\alpha\beta}(\tilde 0') = -\frac{\delta_{\alpha\beta}}{\varepsilon}.
\end{equation}
Combining with Lemma \ref{lem:gra-bdy}, we obtain
\begin{equation*}
|u_{\alpha\beta}(\tilde 0)| = |u_n(\tilde 0)| \cdot |\rho_{\alpha\beta}(\tilde 0')| \leq C_\varepsilon\varepsilon^{-1} \delta_{\alpha\beta}.
\end{equation*}
For the boundary $\partial E_R(0)$, we have
\begin{equation}\label{eq:rho}
	\rho(\tilde{x}')=\frac{1}{2\sqrt{R}}\sum_{\alpha=1}^{n-1}\kappa_\alpha \tilde x_\alpha^2+\frac{1}{R}O(|\tilde{x}'|^3)=\frac{1}{\sqrt R}O(|\tilde x'|^2),
\end{equation}
where $\kappa_1,\cdots,\kappa_{n-1}$ are principal curvatures of $\partial E_1(0)$ at the point $\frac{p}{\sqrt R}$, and so $\frac{\kappa_1}{\sqrt R},\cdots,\frac{\kappa_{n-1}}{\sqrt R}$ are principal curvatures of $\partial E_R(0)$ at $p$.  
Direct computation gives 
\[\frac{\lambda_{\text{min}}(A)}{\sqrt{2\lambda_{\text{max}}(A)R}}\delta_{\alpha\beta}\leq\rho_{\alpha\beta}(\tilde 0') =\frac{\kappa_\alpha}{\sqrt R}\delta_{\alpha\beta} \leq \frac{\lambda_{\text{max}}(A)}{\sqrt{2\lambda_{\text{min}}(A)R}}\delta_{\alpha\beta}.\] 
Using Lemma \ref{lem:gra-bdy} again, we obtain
\[
|u_{\alpha\beta}(\tilde 0)| = |u_n(\tilde 0)| \cdot |\rho_{\alpha\beta}(\tilde 0')|\leq\frac{\lambda_{\text{max}}(A)C_R}{\sqrt{2\lambda_{\text{min}}(A)R}}\delta_{\alpha\beta}\leq C.
\]

\textit{Step 2. Tangential-normal derivatives on $\partial\Omega_R^\varepsilon$.} If $p\in \partial B_\varepsilon(x_0)$, we choose the coordinate such that $p=(x_{0,1},\cdots,x_{0,n-1},x_{0,n}-\varepsilon)$ and consider the angular derivative
	\[
	T_{\alpha} = (x_{\alpha} - x_{0,\alpha})\partial_n - (x_n - x_{0,n})\partial_{\alpha}.
	\]
	In virtue of $S_k(D^2u)=1$, we have
	\[
	L(T_\alpha(u)):=\sum_{i,j} S_k^{ij}(T_{\alpha}(u))_{ij} = 0.
	\]
	It follows from $u\equiv\mathrm{const}$ on $\partial B_\varepsilon(x_0)$ that
	\[
	T_\alpha(u) = (x_\alpha - x_{0,\alpha}) \frac{\partial u}{\partial \nu} \cdot \frac{(x-x_0)_n}{\varepsilon} - (x_n - x_{0,n}) \frac{\partial u}{\partial \nu} \cdot \frac{(x-x_0)_\alpha}{\varepsilon} = 0\quad\text{on }\partial B_\varepsilon(x_0).
	\]
	By \eqref{eq:C1-3}, we find
	\[ |T_{\alpha}(u)| \leq |u_n| + (1+\varepsilon)|u_{\alpha}| \leq 3|Du| \leq C\quad\text{on }\partial B_1(p). \]
	Consider the upper barrier of the form 
	\[ 
	h(x) = A_1(u - \underline{u}^{\varepsilon,R}) + A_2|x - p|^2,
	\]
	where $A_1$ and $A_2$ are positive constants that will be chosen later. Then 
	\[h\geq0=T_\alpha(u)\quad\text{on }\partial B_\varepsilon(x_0),\]
	\[h\geq A_2\geq|T_\alpha(u)|\quad\text{on }\partial B_1(p),\]
	provided that $A_2$ is sufficiently large.
	Recall from \eqref{eq:subsol} that $\underline{u}^{\varepsilon,R}=\underline{u}^{\varepsilon}$ is strictly $k$-convex and $S_k(D^2\underline{u}^\varepsilon)>1$ in $B_2(0)\setminus B_\varepsilon(x_0)$. Then there exists $\tau_0$ sufficiently small such that $\hat{u}(x):=\underline{u}^{\varepsilon}(x)-\tau_0|x|^2$ is $k$-convex and $S_k(D^2\hat{u})\geq1$ in $B_1(p)\setminus B_\varepsilon(x_0)$. By the concavity of $S_k^{1/k}$ in the cone $\Gamma_k$, we have $L(u-\hat{u})\leq0$. A direct computation gives in $B_1(p)\setminus B_\varepsilon(x_0)$,
	\[
	L(u - \underline{u}^{\varepsilon}) = L(u - \hat{u} - \tau_0|x|^2) = L(u - \hat{u}) - 2\tau_0\sum_i S_k^{ii} \leq -2\tau_0\sum_i S_k^{ii},
	\]
	then taking $A_1 = A_2/\tau_0$ yields
  	\[
	L(h \pm T_{\alpha}(u)) = A_1 L(u - \underline{u}^{\varepsilon}) + 2A_2 \sum_i S_k^{ii} \leq -2A_1\tau_0\sum_i  S_k^{ii} + 2A_2 \sum_i S_k^{ii}=0.
	\]
	The maximum principle yields $h \pm T_{\alpha}(u) \geq 0$ in $B_1(p)\setminus B_\varepsilon(x_0)$.
	Consequently, 
	\[\varepsilon|u_{\alpha n}(p)|=  
	|(T_{\alpha}(u))_n(p)| \leq |h_n(p)| = |A_1(u - \underline{u}^{\varepsilon})_n(p)|\leq C|u_n(p)|.
	\]
	By Lemma \ref{lem:gra-bdy}, we obtain
	\begin{equation}\label{eq:TN-1}
		|u_{\alpha n}(p)|\leq C |u_n(p)|\varepsilon^{-1} \leq C \cdot C_\varepsilon\varepsilon^{-1}. 
	\end{equation}

	If $p\in\partial E_R(0)$, we choose the local coordinate at $p$ as in $\textit{Step 1}$ and consider the angular derivative
	\[T_\alpha=\partial_\alpha+\frac{\kappa_\alpha}{\sqrt R}(\tilde x_\alpha\partial_n-\tilde x_n\partial_\alpha). \]
	 The equation $S_k(D^2u)=1$ also implies that
	 \[
	L(T_\alpha(u))=\sum_{i,j} S_k^{ij}(T_{\alpha}(u))_{ij} = 0.
	\] 
	By \eqref{eq:C1-2}, \eqref{eq:tan-d1} and \eqref{eq:rho}, we obtain
	\begin{equation*}
		\begin{split}
			T_\alpha(u) & =u_\alpha+\frac{\kappa_\alpha}{\sqrt R}\tilde{x}_\alpha u_n-\frac{\kappa_\alpha}{\sqrt R}\tilde x_nu_\alpha \\
			& =u_\alpha+\bigg(\rho_\alpha+\frac{1}{R}O(|\tilde x'|^2)\bigg)u_n+\frac{\kappa_\alpha u_\alpha}{R}O(|\tilde x'|^2) \\
			& =\frac{u_n}{R}O(|\tilde x'|^2)+\frac{\kappa_\alpha u_\alpha}{R} O(|\tilde x'|^2) \\
			& =\frac{1}{\sqrt R}O(|\tilde x'|^2), \quad\text{on }\partial E_R(0)\cap E_{\frac{R}{3}}(p).
		\end{split}	
	\end{equation*}
	Also by \eqref{eq:C1-2}, we have
	\[|T_\alpha(u)|\leq C\sqrt{R}\quad\text{on }E_R(0)\cap\partial E_{\frac{R}{3}}(p).\]
	Consider the upper barrier of the form
	\[
	h(x) = -2A_0\bigg(\frac{1}{2}x^TAx-R\bigg)+\frac{A_0}{2}(x-p)^TA(x-p),
	\]
	where \(A_0\) is a positive constant that will be chosen later. Clearly, 
	\[L(h)=-2A_0\sum_{i}S_k^{ii}\lambda_i(A)+A_0\sum_{i}S_k^{ii}\lambda_i(A)\leq0.\]
	Moreover, 
	\[
	h(x) = \frac{A_0}{2}(x-p)^TA(x-p)\geq \frac{\lambda_{\text{min}}(A)A_0}{2}|\tilde x'|^2 \quad\text{on }\partial E_R(0)\cap E_{\frac{R}{3}}(p),
	\]
	and
	\[
	h(x) \geq \frac{A_0}{2}(x-p)^TA(x-p) =\frac{A_0}{3}R \quad\text{on }E_R(0)\cap\partial E_{\frac{R}{3}}(p).
	\] 
	Choosing \(A_0\) sufficiently large, the maximum principle yields \(h\pm \sqrt R T_\alpha(u)\geq 0\) in \(E_R(0)\cap E_{\frac{R}{3}}(p)\). Therefore, 
	\begin{equation}\label{eq:TN-2}
		|u_{\alpha n}(p)|=|(T_\alpha(u))_n(p)|\leq\frac{|h_n(p)|}{\sqrt R} \leq C.
	\end{equation}

	\textit{Step 3. Double normal derivatives on $\partial\Omega_R^\varepsilon$.} For $p\in\partial\Omega_R^\varepsilon$, we choose the local coordinate at $p$ as in \textit{Step 1}, and we further require that $D'^2 u = \{u_{\alpha\beta}\}_{1 \leq \alpha, \beta \leq n-1}$ is diagonal at $p$. It follows that
    \begin{equation}\label{eq:k-H-comp}
	S_{k-1}(D'^2 u)u_{nn} + S_k(D'^2 u) - \sum_{\beta=1}^{n-1} S_{k-2}(D'^2 u )_{\widehat{\beta\beta}}(u_{\beta n})^2 = 1, 
	\end{equation}
	where $(D'^2 u)_{\widehat{\beta\beta}}$ is the $(n-2) \times (n-2)$ matrix obtained by removing the $\beta$-th row and column from $D'^2 u$, and $S_0=1$, $S_{-1}=0$. If $p\in\partial B_\varepsilon(x_0)$, plugging \eqref{eq:TT-I} and \eqref{eq:rho-11} into \eqref{eq:k-H-comp}, we get
	\[
	C_{n-1}^{k-1}u_{nn} + C_{n-1}^k u_{n} \varepsilon^{-1}- C_{n-2}^{k-2}(u_n)^{-1}\varepsilon \sum_{\beta=1}^{n-1} (u_{\beta n})^2= (u_n)^{1-k}\varepsilon^{k-1}.
	\]
	We deduce from Lemma \ref{lem:gra-bdy} and \eqref{eq:TN-1} that 
	\[|u_{nn}(p)|\leq C\cdot C_\varepsilon\varepsilon^{-1}.\]
	If $p\in\partial E_R(0)$, plugging \eqref{eq:TT-I} and \eqref{eq:rho} into \eqref{eq:k-H-comp}, we get
	\[
	S_{k-1}\bigg(\frac{\kappa}{\sqrt R}\bigg)u_{nn} - S_k\bigg(\frac{\kappa}{\sqrt R}\bigg) u_{n}+ \sum_{\beta=1}^{n-1}S_{k-2}\bigg(\frac{\kappa|\beta}{\sqrt R}\bigg)(u_n)^{-1}(u_{\beta n})^2= (-u_n)^{1-k},
	\] 
	where $\kappa=(\kappa_1,\cdots,\kappa_{n-1})$ and $\kappa|\beta=(\kappa_1,\cdots,\kappa_{\beta-1},\kappa_{\beta+1},\cdots,\kappa_{n-1})$.
	Since $\kappa_\alpha$ has positive lower and upper bounds,  it follows from Lemma \ref{lem:gra-bdy} and \eqref{eq:TN-2} that 
	\[|u_{nn}(p)|\leq C.\]
	This completes the proof.
\end{proof}
    
\begin{lemma}\label{lem:C2-est}
	Let $u^{\varepsilon,R}\in C^\infty(\bar{\Omega}_R^\varepsilon)$ be the $k$-convex solution to \eqref{eq:u-e-R}. Then
	\[
	\sup_{\Omega_R^\varepsilon}|D^2u^{\varepsilon,R}|\leq C\cdot C_\varepsilon\varepsilon^{-1},
	\]
	where $C$ is a positive constant depending only on $n$, $k$, $\lambda_{\mathrm{max}}(A)$, $\lambda_{\mathrm{min}}(A)$, $R_0$ and $\alpha$ while independent of $\varepsilon$ and $R$, and $C_\varepsilon$ is as in Lemma \ref{lem:gra-bdy}.
\end{lemma}
\begin{proof}
	Denote \( F(D^2u^{\varepsilon,R}) = S_k^{\frac{1}{k}} (D^2u^{\varepsilon,R}) \). The equation $S_k(D^2u^{\varepsilon,R})=1$ can be rewritten as  
	\[ 
	F(D^2u^{\varepsilon,R}) = 1 \quad \text{in } \Omega_R^\varepsilon. 
	\]  
	Differentiating this with respect to $x_l$ twice gives  
	\[ 
	\sum_{i,j}F^{ij} u^{\varepsilon,R}_{ijll} + \sum_{p,q,r,s}F^{pq,rs} u^{\varepsilon,R}_{pql} u^{\varepsilon,R}_{rsl} = 0. 
	\]  
	In view of the concavity of \( F \)  in the cone $\Gamma_k$, we have 
	\[ 
	\sum_{i,j,l}F^{ij} u^{\varepsilon,R}_{ijll} = \sum_{i,j}F^{ij} (\Delta u^{\varepsilon,R})_{ij} \geq 0. 
	\]  
	By the maximum principle, we obtain  
	\[ 
	\sup_{\Omega_R^\varepsilon} \Delta u^{\varepsilon,R} = \sup_{\partial \Omega_R^\varepsilon} \Delta u^{\varepsilon,R}. 
	\]
	For $k\geq2$, combining with the following equality
	\[
	(\Delta u^{\varepsilon,R})^2 - \sum_{i,j} |u^{\varepsilon,R}_{ij}|^2 = 2S_2(D^2u^{\varepsilon,R}) \geq 0,
	\]
	we obtain
	\begin{equation}\label{eq:mp1}
	\sup_{\Omega_R^\varepsilon}|D^2u^{\varepsilon,R}|\leq C(n)\sup_{\partial \Omega_R^\varepsilon} |D^2u^{\varepsilon,R}|.
	\end{equation}
	For $k=1$, by the equation $\Delta u^{\varepsilon,R}=1$, $u^{\varepsilon,R}_{ij}$ is harmonic. Then
	$$\Delta(|u^{\varepsilon,R}_{ij}|^2)=2|Du^{\varepsilon,R}_{ij}|^2+2u^{\varepsilon,R}_{ij}\Delta u^{\varepsilon,R}_{ij}\geq0,$$
	and thus $\Delta(|D^2u^{\varepsilon,R} |^2)\geq0$.
	The maximum principle yields
	\begin{equation}\label{eq:mp2}
	\sup_{\Omega_R^\varepsilon}|D^2u^{\varepsilon,R}|=\sup_{\partial \Omega_R^\varepsilon} |D^2u^{\varepsilon,R}|.
	\end{equation}
	With \eqref{eq:mp1} and \eqref{eq:mp2}, the proof is concluded by invoking Lemma \ref{lem:C2-est}.
\end{proof}

\section{Convergence and Hessian measures}\label{sec:u-e}
In this section, we investigate the limit of the solution $u^{\varepsilon,R}$ to the approximating problem \eqref{eq:u-e-R} as $R\to\infty$ and $\varepsilon\to0$ sequentially.
We first observe that $u^{\varepsilon,R}$ is increasing with respect to $R$. Indeed, for $R_2>R_1$, we have
\begin{equation*}
	\begin{cases}
		S_k(D^2u^{\varepsilon,R_1})=S_k(D^2u^{\varepsilon,R_2})=1 & \text{in }\Omega_{R_1}^\varepsilon, \\
		u^{\varepsilon,R_1}=u^{\varepsilon,R_2}=0 & \text{on }\partial B_\varepsilon(x_0), \\
		u^{\varepsilon,R_1}=\underline{u}=\underline{u}^{\varepsilon,R_2}\leq u^{\varepsilon,R_2} & \text{on }\partial E_{R_1}(0),
		\end{cases}
\end{equation*}
where $\underline{u}^{\varepsilon,R_2}$ is given by \eqref{eq:subsol}. The comparison principle implies that $u^{\varepsilon,R_1}\leq u^{\varepsilon,R_2}$ in $\Omega^\varepsilon_{R_1}$. Lemma \ref{lem:C0-est} allows   us to define $x\in\mathbb{R}^n\setminus B_\varepsilon(x_0)$, 
$$u^\varepsilon(x)=\lim_{R\to\infty}u^{\varepsilon,R}(x).$$ 
By Lemma \ref{lem:C0-est}, Remark \ref{rmk:Du-est} and Lemma \ref{lem:C2-est}, $u^{\varepsilon,R}$, $|Du^{\varepsilon,R}|$ and $|D^2u^{\varepsilon,R}|$ are locally uniformly bounded in $\mathbb{R}^n\setminus B_\varepsilon(x_0)$, independent of $R$. Therefore, by applying Evans-Krylov theorem and Schauder estimates, we can obtain higher order estimates for $u^{\varepsilon,R}$ that are independent of $R$. We conclude that as $R\to\infty$,
\[u^{\varepsilon,R}\to u^\varepsilon\quad\text{in }C_{\mathrm{loc}}^\infty(\mathbb{R}^n\setminus B_\varepsilon(x_0)).\] 
It follows that $u^\varepsilon\in C^\infty(\mathbb{R}^n\setminus B_\varepsilon(x_0))$ is $k$-convex and satisfies
\[
\begin{cases}
	S_k(D^2u^\varepsilon)=1 & \text{in }\mathbb{R}^n\setminus\bar{B}_\varepsilon(x_0), \\
	u^\varepsilon=0 & \text{on }\partial B_\varepsilon(x_0).
\end{cases}
\]
Recall from the proof of Proposition \ref{prop:u-e-R-exist} that $\underline{u}^{\varepsilon,R}\geq\underline{u}$ in $\mathbb{R}^n\setminus B_1(0)$, together with Lemma \ref{lem:C0-est}, we have
\begin{equation}\label{eq:u-ue-psi}
	\underline{u}\leq u^{\varepsilon}\leq \psi\quad\text{in }\mathbb{R}^n\setminus B_1(0).
\end{equation}
Combining with \eqref{eq:subsol-u-2}, $u^\varepsilon$ satisfies 
\[
\liminf_{|x|\to\infty}\bigg(u^\varepsilon(x)-\bigg(\frac{1}{2}x^TAx+\mu(\alpha)\bigg)\bigg)\geq\lim_{|x|\to\infty}\bigg(\underline{u}(x)-\bigg(\frac{1}{2}x^TAx+\mu(\alpha)\bigg)\bigg)=0,
\]
\[
\limsup_{|x|\to\infty}\bigg(u^\varepsilon(x)-\bigg(\frac{1}{2}x^TAx+\mu(\alpha)\bigg)\bigg)\leq\lim_{|x|\to\infty}\bigg(\psi(x)-\bigg(\frac{1}{2}x^TAx+\mu(\alpha)\bigg)\bigg)=0.
\]
Let $\alpha_0$ be as in Proposition \ref{prop:u-e-R-exist}. By letting 
\begin{equation}\label{eq:c*-c}
	c^*=\mu(\alpha_0)>0\quad\text{and}\quad c=\mu(\alpha)\geq c^*,
\end{equation} 
with $\alpha\geq\alpha_0$, we obtain that $u^\varepsilon$ is the solution to the exterior Dirichlet problem for $k$-Hessian equations with the asymptotic behavior $\frac{1}{2}x^TAx+c$. Precisely,
\begin{proposition}\label{prop:u-e}
	Let $n\geq3$ and $1\leq k\leq n$. For any given diagonal matrix $A\in\mathcal{A}_k$, assuming further that $\lambda_{\mathrm{max}}(A)<\frac{1}{2} $ when $k=1$, there exists a positive constant $c^*$, depending only on $n$, $k$, $\lambda_{\mathrm{max}}(A)$ and $\lambda_{\mathrm{min}}(A)$, such that for every $c\geq c^*$, the exterior Dirichlet problem 
	\begin{equation}\label{eq:EDP-u-e}
	\begin{cases}
		S_k(D^2u)=1 \quad\text{in }\mathbb{R}^n\setminus B_\varepsilon(x_0), \\
		u=0 \quad\text{on }\partial B_{\varepsilon}(x_0), \\
		\lim\limits_{|x|\to\infty}(u(x)-(\frac{1}{2}x^TAx+c))=0, 
	\end{cases}
\end{equation}
	has a unique $k$-convex solution  $u^\varepsilon\in C^\infty(\mathbb{R}^n\setminus B_\varepsilon(x_0))$. Moreover, for any given bounded domain $\Omega$ with $\bar{\Omega}\subset\mathbb{R}^n\setminus\{x_0\}$, $u^\varepsilon$ satisfies the uniform estimate 
	\[\|u^\varepsilon\|_{C^1(\Omega)}\leq C(\Omega),\]
	for any $0<\varepsilon<\frac{1}{2}\min\{\mathrm{dist}(x_0,\partial\Omega),1\}$, where $C(\Omega)$ is a positive constant depending only on $n$, $k$, $\lambda_{\mathrm{max}}(A)$, $\lambda_{\mathrm{min}}(A)$, $c$ and $\mathrm{dist}(x_0, \partial\Omega)$ while independent of $\varepsilon$.
\end{proposition}
\begin{proof}
	The existence is a consequence of the above argument, and the uniqueness follows from the comparison principle. By Lemma \ref{lem:C0-est}, we have
	\begin{equation}\label{eq:u-e-C0}
		\underline{u}^{\varepsilon,R}\leq u^\varepsilon\leq\psi\quad\text{in }\mathbb{R}^n\setminus B_\varepsilon(x_0),
	\end{equation}
	 and $\underline{u}^{\varepsilon,R}$ clearly is locally bounded below in $\mathbb{R}^n\setminus B_\varepsilon(x_0)$ independent of $\varepsilon$ and $R$. The estimate is then obtained via Remark \ref{rmk:Du-est}.
\end{proof}

In the rest of this section, we will always assume \eqref{eq:c*-c}. we denote   
\begin{equation*}
	\tilde{u}^\varepsilon(x)=
	\begin{cases}
		u^\varepsilon(x), & x\in\mathbb{R}^n\setminus B_\varepsilon(x_0), \\
		0, & x\in B_\varepsilon(x_0).
	\end{cases}
\end{equation*}
Next, we will show that $\tilde{u}^\varepsilon$ converges locally in $k$-Hessian measure sense to the prescribed quadratic polynomial $\frac{1}{2}x^TAx+c$, i.e. the function $\psi(x)$ defined in \eqref{eq:psi}, as $\varepsilon\to0^+$, by following the framework of $k$-Hessian measures developed in \cite{Trudinger-Wang-II, Trudinger-Wang-III} and adapted in \cite{Wang-Xiao}.
\begin{proposition}\label{prop:limit}
	There holds $\tilde{u}^\varepsilon\to\psi$ in $L^1_{\mathrm{loc}}(\mathbb{R}^n)\cap C^{0,\gamma}_{\mathrm{loc}}(\mathbb{R}^n\setminus\{x_0\})$ for any $0<\gamma<1$ as $\varepsilon\to0^+$, where $\psi(x)=\frac{1}{2}x^TAx+c$. 
\end{proposition}


We observe that $\tilde{u}^{\varepsilon}$ is decreasing in $\mathbb{R}^n$ with respect to $\varepsilon$. Indeed, since $u^{\varepsilon,R}$ is subharmonic, the maximum principle gives $u^{\varepsilon,R}\geq0$ in $\Omega_R^\varepsilon$. For $\varepsilon_1<\varepsilon_2$, we have 
\begin{equation*}
	\begin{cases}
		S_k(D^2u^{\varepsilon_1,R})=S_k(D^2u^{\varepsilon_2,R})=1 & \text{in }\Omega_R^{\varepsilon_2}, \\
		u^{\varepsilon_1,R}\geq 0=u^{\varepsilon_2,R} & \text{on }\partial B_{\varepsilon_2}(x_0), \\
		u^{\varepsilon_1,R}=\underline{u}=u^{\varepsilon_2,R} & \text{on }\partial E_R(0).
	\end{cases} 
\end{equation*} 
The comparison principle implies $u^{\varepsilon_1,R}\geq u^{\varepsilon_2,R}$ in $\Omega_R^{\varepsilon_2}$. Sending $R\to\infty$, $u^{\varepsilon_1}\geq u^{\varepsilon_2}$ in $\mathbb{R}^n\setminus B_{\varepsilon_2}(x_0)$. Notice that $u^{\varepsilon_1}\geq0=\tilde{u}^{\varepsilon_2}$ in $B_{\varepsilon_2}(x_0)\setminus B_{\varepsilon_1}(x_0)$, and thus $\tilde{u}^{\varepsilon_1}\geq\tilde{u}^{\varepsilon_2}$ in $\mathbb{R}^n$. In view of \eqref{eq:u-e-C0} and $\tilde{u}^\varepsilon=0$ in $B_\varepsilon(x_0)$,
the dominated convergence theorem, combined with the locally uniformly $C^1$ estimate in Proposition \ref{prop:u-e}, ensures that there exists $v\in L^1_{\mathrm{loc}}(\mathbb{R}^n)\cap C^{0,1}_{\mathrm{loc}}(\mathbb{R}^n\setminus\{x_0\})$ such that, as $\varepsilon\to 0^+$, for any $0<\gamma<1$, 
\begin{equation}\label{eq:u-e-v}
	\tilde{u}^\varepsilon\to v\quad\text{in }L^1_{\mathrm{loc}}(\mathbb{R}^n)\cap C^{0,\gamma}_{\mathrm{loc}}(\mathbb{R}^n\setminus\{x_0\}).
\end{equation}
To guarantee the upper semi-continuity of $v$ at $x_0$, we additionally define 
$$v(x_0):=\limsup_{x\to x_0}v(x).$$
Clearly, $v(x_0)$ is finite.

To further investigate properties of $\tilde{u}^\varepsilon$ and $v$ in the whole space $\mathbb{R}^n$, particularly across the singularity point $x_0$, we introduce an extended definition of $k$-convexity to upper semi-continuous functions by following \cite{Trudinger-Wang-II, Trudinger-Wang-III}. An upper semi-continuous function $u: \Omega\to[-\infty,\infty)$ is called $k$-convex in $\Omega$ if $S_k(D^2q)\geq0$ for any quadratic polynomial $q$ such that $u-q$ has a finite local maximum in $\Omega$. Furthermore, a $k$-convex function is called proper if it is not identically $-\infty$ on any connected component of $\Omega$. We denote the class of proper $k$-convex functions in $\Omega$ by $\Phi^k(\Omega)$.

\begin{lemma}\label{lem:Phi-k}
	Let $\tilde{u}^\varepsilon$ and $v$ be defined as above. Then $\tilde{u}^\varepsilon$, $v\in\Phi^k(\mathbb{R}^n)$.
\end{lemma}
\begin{proof}
	It is clear that $\tilde{u}^\varepsilon$ and $v$ are upper semi-continuous and proper in $\mathbb{R}^n$. 

	Let $q$ be a quadratic polynomial such that $\tilde{u}^\varepsilon-q$ has a finite local maximum at some point $\bar{x}\in\mathbb{R}^n  $. If $\bar{x}\in\mathbb{R}^n\setminus\bar{B}_\varepsilon(x_0)$ or $B_\varepsilon(x_0)$, then $S_k(D^2q)=S_k(D^2q(\bar{x}))\geq0$ due to the $k$-convexity and $C^2$ regularity of $\tilde{u}^\varepsilon$ at $\bar{x}$. If $\bar{x}\in\partial B_\varepsilon(x_0)$, we additionally set $\underline{u}^{\varepsilon,R}:=\underline{u}^\varepsilon\leq0$ in $B_\varepsilon(x_0)$. 
	Recalling from \eqref{eq:subsol} that $\underline{u}^{\varepsilon,R}\in C^{\infty}(\mathbb{R}^n\setminus B_\varepsilon(x_0))$ satisfies $\underline{u}^{\varepsilon,R}=\underline u^\varepsilon$ in $B_2(0)\setminus B_\varepsilon(x_0)$, this yields that $\underline{u}^{\varepsilon,R}$ is still smooth and $k$-convex in $\mathbb{R}^n$. Since $u^{\varepsilon,R}$ is increasing with respect to $R$, we have
	$$\tilde{u}^\varepsilon\geq u^{\varepsilon,R}\geq\underline{u}^{\varepsilon,R}\ \text{in }E_R(0)\setminus B_\varepsilon(x_0)\quad\text{and}\quad\tilde{u}^\varepsilon\geq\underline{u}^{\varepsilon,R}\text{ in }B_\varepsilon(x_0).$$ 
	Then for some $\delta>0$ with $B_\delta(\bar{x})\subset E_R(0)  $, 
	$$\tilde{u}^\varepsilon\geq\underline{u}^{\varepsilon,R}\text{ in }B_\delta(\bar{x})\quad\text{and}\quad\tilde{u}^\varepsilon(\bar{x})=\underline{u}^{\varepsilon,R}(\bar{x}).$$
	This yieids $\underline{u}^{\varepsilon,R}-q$ also attains a finite local maximum at $\bar{x}$. Utilizing the $k$-convexity of $C^2$ funciton $\underline{u}^{\varepsilon,R}$, we have $S_k(D^2q)\geq0$. That is, $\tilde{u}^\varepsilon\in\Phi^k(\mathbb{R}^n)$.
	
	We prove $v\in\Phi^k(\mathbb{R}^n)$ by contradiction. If not, there exists a quadratic polynomial $q$, a point $\bar{x}$ and some constant $\delta>0$ such that
	\[v-q\leq0 \text{ in }B_\delta(\bar{x})\quad\text{and}\quad v(\bar{x})-q(\bar{x})=0,\]
	and
	\[S_k(D^2q)<0.\]
	Without loss of generality, we may further assume
	\[
	v - q < c_\delta < 0 \text{ on } \partial B_\delta(\bar{x}). 
	\]
	As otherwise, we can replace the above $q$ with the perturbed polynomial \(q_\beta(x):= q(x) + \beta|x - \bar{x}|^2\) where \(\beta > 0\) is chosen so small that \(S_k(D^2q_\beta) < 0\). By the monotonicity of $\tilde{u}^\varepsilon$ with respect to $\varepsilon$, we obtain 
	\begin{equation}\label{eq:u-q}
		\tilde{u}^\varepsilon - q \leq v - q < c_\delta < 0 \quad \text{on} \, \partial B_\delta(\bar{x}).
	\end{equation} 
	We next consider two cases.

	\textit{Case 1.} If $\bar{x}\neq x_0$, we may assume \(\bar{B}_\delta(\bar{x})\cap\bar{B}_\varepsilon(x_0)=\emptyset\) for sufficiently small \(\varepsilon > 0\). 
	By \eqref{eq:u-e-v}, \(\tilde{u}^\varepsilon(\bar{x}) \to v(\bar{x})\) as $\varepsilon\to 0^+$. That is for any \(\eta > 0\) there exists \(\varepsilon_\eta > 0\) such that when \(0<\varepsilon < \varepsilon_\eta\), \(|\tilde{u}^\varepsilon(\bar{x}) - v(\bar{x})| < \eta\). Taking $\eta=-\frac{c_\delta}{2}$, we get
	\begin{equation*}
		\tilde{u}^{\varepsilon}(\bar{x})>v(\bar{x})-\eta=q(\bar{x})-\eta=q(\bar{x})+\frac{c_\delta}{2}.
	\end{equation*}
	Combining with \eqref{eq:u-q}, \(\tilde{u}^\varepsilon - q\) attains  a finite local maximum in \(B_\delta(\bar{x})\). Since \(S_k(D^2\tilde{u}^\varepsilon) = 1\) in \(B_\delta(\bar{x})\), we have \(S_k(D^2q) \geq 1\), a contradiction.

	\textit{Case 2.} If $\bar{x}=x_0$, we deduce from \( v(x_0) = \limsup_{x \to x_0} v(x) \)   that for any \( \eta > 0 \) small, there exists \(\{x_n\} \subset B_\delta(x_0) \setminus \{x_0\}\) and \( x_n \to x_0 \) such that \(|v(x_n) - v(x_0)| < \frac{\eta}{3}\) and \(|q(x_n)-q(x_0)|<\frac{\eta}{3}\). We fix \( x_n \), then there exists \( \varepsilon_1 = \varepsilon_1(\eta, x_n) > 0 \) such that when \( 0<\varepsilon < \varepsilon_1 \), \( |\tilde{u}^\varepsilon(x_n) - v(x_n)| < \frac{\eta}{3} \). Therefore, for this \( x_n \in B_\delta(x_0)\setminus\{x_0\} \) and $0<\varepsilon<\varepsilon_1$,
	\[
	|\tilde{u}^\varepsilon(x_n) - q(x_n)| \leq |\tilde{u}^\varepsilon(x_n) - v(x_n)| + |v(x_n) - v(x_0)| + |q(x_0) - q(x_n)| < \eta:=-\frac{c_\delta}{2}.
	\]
	Combining with \eqref{eq:u-q}, \( \tilde{u}^\varepsilon - q \) also attains a finite local maximum in \( B_\delta(x_0) \). Since $\tilde{u}^\varepsilon\in\Phi^k(\mathbb{R}^n)$, we have \( S_k(D^2q) \geq 0 \), a contradiction. 
\end{proof}

We proceed to investigate the convergence of $\tilde{u}^{\varepsilon}$ to $v$ in the sense of $k$-Hessian measure. As $\tilde{u}^\varepsilon\in C^{0}(\mathbb{R}^n)$, we can define the mollification of $\tilde{u}^\varepsilon$ by
\[\tilde{u}^\varepsilon_h(x)=\int_{\mathbb{R}^n}\rho_{h}(x-y)\tilde{u}^\varepsilon(y)\mathrm{d}y=\int_{\mathbb{R}^n}h^{-n}\rho\bigg(\frac{x-y}{h}\bigg)\tilde{u}^\varepsilon(y)\mathrm{d}y,\quad x\in\mathbb{R}^n,\]
where $\rho$ is the standard mollifier, namely $\rho$ is a smooth, nonnegative function with support in $B_1(0)$, and 
$\int_{B_1(0)}\rho=1$. From Lemma \ref{lem:Phi-k} and Lemma 2.3 of \cite{Trudinger-Wang-II}, it follows that $\tilde{u}_h^\varepsilon\in C^\infty(\mathbb{R}^n)\cap\Phi^k(\mathbb{R}^n)$ and as $h\to0$,
\begin{equation}\label{eq:u-h-u}
	\tilde{u}_h^\varepsilon\to \tilde{u}^\varepsilon\quad\text{in }C^0_{\mathrm{loc}}(\mathbb{R}^n)\cap C^\infty_{\mathrm{loc}}(\mathbb{R}^n\setminus\bar{B}_\varepsilon(x_0))\cap C^\infty_{\mathrm{loc}}(B_\varepsilon(x_0)).
\end{equation}
Assume $\varepsilon_m, h_m \to 0$ and  $h_m \ll \varepsilon_m$ as $m \to \infty$, and denote  
$u^m := \tilde{u}_{h_m}^{\varepsilon_m}$ for simplicity. Via Lemma \ref{lem:C0-est} and Remark \ref{rmk:Du-est}, we have as $m\to\infty$, 
$$u^m-\tilde{u}^{\varepsilon_m}\to 0\quad\text{in }C_{\text{loc}}^0(\mathbb{R}^n) \cap C_{\text{loc}}^{0,1}(\mathbb{R}^n\setminus \{x_0\}).$$
Combining \eqref{eq:u-e-v}, this yields that as $m\to\infty$, for any $0<\gamma<1$,
$$u^m\to v \quad\text{in }L_{\text{loc}}^1(\mathbb{R}^n) \cap C_{\text{loc}}^{0,\gamma}(\mathbb{R}^n\setminus \{x_0\}).$$ 
By the equivalence of convergence in $L^1_{\mathrm{loc}}$ and convergence in $k$-Hessian measure $\mu_k$ and Theorem 1.1 of \cite{Trudinger-Wang-II}, $\mu_k[u^m]$ converges weakly to $\mu_k[v]$, where $\mu_k[u^m]$ and $\mu_k[v]$ are $k$-Hessian measures generated by $u^m$ and $v$ respectively. According to Portmanteau
theorem, this can be equivalently stated as closed and open set conditions, i.e. for any ball $B = B_r(x)$,    
\begin{equation}\label{eq:oc}
\liminf_{m \to \infty} \mu_k[u^m](B) \geq \mu_k[v](B),
\end{equation}
and  
\begin{equation}\label{eq:cc}
\limsup_{m \to \infty} \mu_k[u^m](\bar{B}) \leq \mu_k[v](\bar{B}).
\end{equation}
\begin{lemma}\label{lem:mu-k}
	Let $1\leq k\leq \frac{n}{2}$ and $\mu_k[v]$ be as above, then $\mu_k[v]$ is the standard measure in $\mathbb{R}^n$. That is, for any ball $B=B_r(x)$, we have 
$$\mu_k[v](B)=\int_B\mathrm{d}x=|B|.$$ 
\end{lemma}  
\begin{proof}
Applying Theorem 1.1 in \cite{Trudinger-Wang-II} to $u^m\in C^\infty(\mathbb{R}^n)$, we have
\[
\begin{split}
\mu_k[u^m](B) & =\int_{B}S_k(D^2u^m)\mathrm{d}x \\
& = \int_{I_1} S_k(D^2u^m) \mathrm{d}x + \int_{I_2} S_k(D^2u^m) \mathrm{d}x + \int_{I_3} S_k(D^2u^m) \mathrm{d}x,
\end{split}
\]
where $I_1=B \setminus B_{\frac{3}{2}\varepsilon_m}(x_0)$, $I_2=B \cap \big(B_{\frac{3}{2}\varepsilon_m}(x_0) \setminus \bar{B}_{\frac{1}{2}\varepsilon_m}(x_0)\big)$, $I_3=B \cap \bar{B}_{\frac{1}{2}\varepsilon_m}(x_0)$. We will enlarge the three items in RHS separately.

We fix $m$. By the convergence in \eqref{eq:u-h-u},  for any \( \eta > 0 \), there exists \( h_{\eta} > 0 \) such that when \( 0 < h_m < h_{\eta} \), we have for $l=0, 1, 2$,
\[
|D^lu^m - D^l\tilde{u}^{\varepsilon_m}| < \eta \quad \text{in }  \big(\bar{B} \setminus B_{\frac{4}{3}\varepsilon_m}(x_0)\big) \cup \bar{B}_{\frac{2}{3}\varepsilon_m}(x_0).
\]
Then $S_k(D^2u^m)\leq S_k(D^2 \tilde{u}^{\varepsilon_m} + \eta I)$ in $I_1\cup I_3$. Setting $S_0(\cdot):=1$, a direct computation gives the expansion 
$$S_k(D^2 \tilde{u}^{\varepsilon_m} + \eta I) = S_k(D^2 \tilde{u}^{\varepsilon_m}) + \sum_{j=0}^{k-1} C_{n-j}^{k-j} \eta^{k-j} S_j(D^2 \tilde{u}^{\varepsilon_m}).$$ 
Combining with $S_k(D^2\tilde{u}^{\varepsilon_m})=1$ in $\mathbb{R}^n\setminus\bar{B}_{\varepsilon_m}(x_0)$, we obtain
\begin{equation}\label{eq:I1}
\begin{aligned}
\int_{I_1} S_k(D^2 u^m) \mathrm{d}x &\le \int_{I_1} S_k(D^2 \tilde{u}^{\varepsilon_m} + \eta I) \mathrm{d}x \\
&\le \int_{I_1} S_k(D^2 \tilde{u}^{\varepsilon_m}) \mathrm{d}x + C_1 \sum_{j=0}^{k-1} \eta^{k-j} \int_{I_1} S_j(D^2 \tilde{u}^{\varepsilon_m}) \mathrm{d}x \\
&= |I_1| + C_1 \sum_{j=0}^{k-1} \eta^{k-j} \int_{I_1} S_j(D^2 \tilde{u}^{\varepsilon_m}) \mathrm{d}x \\
&\le \begin{cases} 
|B| + C_1 \eta \varepsilon_m^{-2(k-1)}, \quad \text{if } k < \frac{n}{2}, \\[1ex]
|B| + C_1 \eta \varepsilon_m^{-2(k-1)} |\log \varepsilon_m|^{-(k-1)}, \quad \text{if } k = \frac{n}{2}. 
\end{cases}
\end{aligned}
\end{equation}
The last inequality comes from Lemma \ref{lem:C2-est}. 
Since $D^2\tilde{u}^{\varepsilon_m}=0$ in $B_{\varepsilon_m}(x_0)$,  we have
\begin{equation}\label{eq:I3}
	\int_{I_3} S_k(D^2u^m) \mathrm{d}x \leq \int_{I_3}S_k(D^2 \tilde{u}^{\varepsilon_m} + \eta I) \mathrm{d}x = \int_{I_3}S_k(\eta I) \mathrm{d}x\leq C_3\eta^k\varepsilon_m^n.
\end{equation}
Using the divergence structure of $S_k$, namely,
\[S_k(D^2u^m)=\frac{1}{k}\sum_{j}\frac{\partial}{\partial x_j}\bigg(\sum_i   S_k^{ij}(D^2u^m)\frac{\partial u^m}{\partial x_i}\bigg),\]
together with Lemmas \ref{lem:C1_est} and \ref{lem:C2-est}, we obtain
\begin{equation}\label{eq:I2}
\begin{split}
    \int_{I_2} S_k(D^2u^m) \mathrm{d}x &\leq \int_{B_{\frac{3}{2}\varepsilon_m}(x_0) \setminus \bar{B}_{\frac{1}{2}\varepsilon_m}(x_0)}S_k(D^2u^m)  \mathrm{d}x \\
    &= \frac{1}{k}\sum_{i,j}\bigg(\int_{\partial B_{\frac{3}{2}\varepsilon_m}(x_0)} S_k^{ij} (u^m)_i \nu_j \mathrm{d}\sigma - \int_{\partial B_{\frac{1}{2}\varepsilon_m}(x_0)} S_k^{ij} (u^m)_i \nu_j \mathrm{d}\sigma\bigg) \\
    &\leq C_2|\partial B_{\frac{3}{2}\varepsilon_m}(x_0)|\sup_{\partial B_{\frac{3}{2}\varepsilon_m}(x_0)}|D^2u^m|^{k-1}|Du^m|\\
    &\leq C_2|\partial B_{\frac{3}{2}\varepsilon_m}(x_0)|\sup_{\partial B_{\frac{3}{2}\varepsilon_m}(x_0)}(|D^2\tilde{u}^{\varepsilon_m}|+\eta)^{k-1}(|Du^m|+\eta)\\
    &\leq 
    \begin{cases}
    C_2 \varepsilon_m^{n-1}(\varepsilon_m^{-2}+\eta)^{k-1}(\varepsilon_m^{-1}+\eta),\quad\text{if }k<\frac{n}{2}, \\[1.5ex]
    C_2 \varepsilon_m^{n-1}(\varepsilon_m^{-2}|\log\varepsilon_m|^{-1}+\eta)^{k-1}(\varepsilon_m^{-1}|\log\varepsilon_m|^{-1}+\eta), \quad\text{if }k=\frac{n}{2},
    \end{cases}
\end{split}
\end{equation}
where $\nu$ is the unit outward normal vector of $B_{\frac{3}{2}\varepsilon_m}(x_0) \setminus \bar{B}_{\frac{1}{2}\varepsilon_m}(x_0)$. We conclude from \eqref{eq:I1}-\eqref{eq:I2} that
\[
\mu_k[u^m](B) \leq \begin{cases}
|B| + C_4 \varepsilon_m^{n-2k} + C_5 \eta \varepsilon_m^{-2(k-1)}, \quad\text{if }k<\frac{n}{2}, \\[1.5ex]
|B| + C_4 \varepsilon_m^{n-2k} |\log\varepsilon_m|^{-k}+ C_5 \eta \varepsilon_m^{-2(k-1)}|\log\varepsilon_m|^{-(k-1)}, \quad\text{if }k=\frac{n}{2},
\end{cases}
\]
for some \( C_4\), \(C_5 > 0 \) that are independent of $m$ and the choice of \( B \). Taking \( \eta \leq \varepsilon_m^{2k} \) and sending $m\to\infty$,  in view of \eqref{eq:oc}, we obtain
\[ \mu_k[v](B) \leq |B|.\] 
Similarly, in view of \eqref{eq:cc}, we can show that for any \( \theta \in (0, 1) \), $
\mu_k[v](\bar{B}_{\theta r}(x)) \geq |\bar{B}_{\theta r}(x)| $.
Letting \( \theta \to 1^- \), we obtain 
\[ \mu_k[v](B) \geq |B|. \]
Therefore, $\mu_k[v]$ is the standard measure in $\mathbb{R}^n$.
\end{proof}


Now, we are in a position to complete the proof of Proposition \ref{prop:limit}. A key point is to illustrate the continuity of $v$ at $x_0$ and then invoke the comparison principle for continuous $k$-convex functions.

\begin{proof}[Proof of Proposition \ref{prop:limit}.]
	By \eqref{eq:u-ue-psi} and \eqref{eq:u-e-v}, we have $\underline{u}\leq v\leq \psi$ in $\mathbb{R}^n\setminus B_1(0)$.
This implies that $v$ satisfies the asymptotic behavior at infinity
$$\lim_{|x|\to\infty}\bigg(v(x)-\bigg(\frac{1}{2}x^TAx+c\bigg)\bigg)=0.$$ 
Notice that $\mu_k[\psi](B)=\int_{B}S_k(D^2\psi)\mathrm{d}x=|B|$ for any $B=B_r(x)$. Via Lemma \ref{lem:mu-k}, we have 
\begin{equation*}
	\begin{cases}
		\mu_k[v]=\mu_k[\psi]\quad\text{in }\mathbb{R}^n, \\
		\lim\limits_{|x|\to\infty}(v-\psi)(x)=0.
	\end{cases}
\end{equation*}
We claim that $v\equiv \psi$ in $\mathbb{R}^n$. Indeed, for any $\eta>0$, there exists $\tilde{R}$ such that 
$$|\psi-v|<\eta,\quad\text{for }|x|\geq\tilde{R}.$$
Thanks to $\mu_k[v]$ being the standard measure in $\mathbb{R}^n$, we derive for $0<r<1$, 
$$\mu_{k}[v](B_r(x))=C(n)r^n\leq C(n)r^{n-2k+1}.$$
According to the equivalence between $k$-convexity and local H\"{o}lder continuity (see for instance (4.4) in \cite{Trudinger-Wang-III}), $v$ is locally H\"{o}lder continuous in $\mathbb{R}^n$. This allows us to apply the comparison principle, which is provided in Theorem 3.1 of \cite{Trudinger-Wang-I},
$$\psi-\eta<v<\psi+\eta\quad\text{for }|x|<\tilde{R}.$$
Therefore, $\psi-\eta\leq v\leq\psi+\eta$ in $\mathbb{R}^n$. By sending $\eta\to0$, the claim is thus proved, thereby completing the proof by virtue of \eqref{eq:u-e-v}.
\end{proof}

\section{Proofs of Theorems \ref{thm:main} and \ref{thm:main-1}}\label{sec:proof}
We begin with a special and simple case of Theorems \ref{thm:main} and \ref{thm:main-1}, stated as the following lemma, where we additionally assume that the matrix $A\in\mathcal{A}_k$ is diagonal and the vector $b$ vanishes.

\begin{lemma}\label{lem:AIb0}
	Let $n\geq3$ and $1\leq k \leq \frac{n}{2}$. Then for any given diagonal matrix $A\in\mathcal{A}_k$, assuming further that $\lambda_{\mathrm{max}}(A)<\frac{1}{2} $ when $k=1$, there exists a smooth, strictly convex domain $\Omega_0$ and a positive   constant $c^*$, depending only on $n$, $k$, $\lambda_{\mathrm{max}}(A)$ and $\lambda_{\mathrm{min}}(A)$, such that for every $c\geq c^*$, the problem 
	\begin{equation*}
	\begin{cases}
		S_k(D^2u)=1  \quad\text{in }\mathbb{R}^n\setminus\bar\Omega_0, \\
		u=0  \quad\text{on }\partial\Omega_0, \\
		\lim\limits_{|x|\to\infty} \big(u(x)-\big(\frac{1}{2}x^TAx+c\big)\big)=0,
	    \end{cases}
\end{equation*}
	has a unique $k$-convex solution $u\in C^\infty(\mathbb{R}^n\setminus\Omega_0)$ which is not quasiconvex.
\end{lemma}

To prove Theorem \ref{thm:main} and Theorem \ref{thm:main-1}, it suffices to prove Lemma \ref{lem:AIb0}. Indeed, suppose that $A\in\mathcal{A}_k$ and $b\in\mathbb{R}^n$. Consider the decomposition $A=P^T\Lambda P$, where $P$ is an orthogonal matrix and $\Lambda$ is a diagonal matrix with $\lambda(\Lambda)=\lambda(A)$. It is clear that $\frac{1}{2}x^TAx+b\cdot x$ takes its minimum in $\mathbb{R}^n$ at $y_0:=-A^{-1}b$. Let
\[\hat{x}=P(x-y_0).\]
By Lemma \ref{lem:AIb0}, there exists a smooth, strictly convex domain $\hat{\Omega}_0$ and a constant $\hat{c}^*$ depending only on $n$, $k$ and $\Lambda$ such that for every $\hat{c}\geq\hat{c}^*$, there exists a unique $k$-convex solution $\hat{u}\in C^\infty(\mathbb{R}^n\setminus\hat{\Omega}_0)$ to the problem 
\begin{equation*}
	\begin{cases}
		S_k(D^2\hat{u})=1  \quad\text{in }\mathbb{R}^n\setminus\bar{\hat{\Omega}}_0, \\
		\hat{u}=0  \quad\text{on }\partial\hat{\Omega}_0, \\
		\lim\limits_{|\hat{x}|\to\infty} \big(\hat{u}(\hat{x})-\big(\frac{1}{2}\hat{x}^T\Lambda \hat{x}+\hat{c}\big)\big)=0,
	    \end{cases}
\end{equation*}
while $\hat{u}$ is not quasiconvex. 
Let
\[\Omega_0=P^{-1}\hat{\Omega}_0=\{P^{-1}\hat{x}: \hat{x}\in\hat{\Omega}_0\},\quad c=\hat{c}+\frac{1}{2}b^TA^{-1}b,\quad c^*=\hat{c}^*+\frac{1}{2}b^TA^{-1}b,\]
and
\[u(x)=\hat{u}(\hat x)=\hat{u}(P(x-y_0)).\]
Then a direct computation shows 
$$S_k(D^2u(x))=S_k(P^TD^2\hat{u}(\hat{x})P)=1,\quad x\in \mathbb{R}^n\setminus\bar{\Omega}_0,$$ 
$$u(x)=0,\quad x\in\partial\Omega_0,$$ 
and
\[
\begin{split}
	0= & \lim_{|\hat{x}|\to\infty}\bigg(\hat{u}(\hat{x})-\bigg(\frac{1}{2}\hat{x}^T\Lambda\hat{x}+\hat{c}\bigg)\bigg) \\
	= & \lim_{|x|\to\infty}\bigg(u(x)-\bigg(\frac{1}{2}(x-y_0)^TP^T\Lambda P(x-y_0)+\hat{c}\bigg)\bigg)\\ 
	= & \lim_{|x|\to\infty}\bigg(u(x)-\bigg(\frac{1}{2}x^TAx-y_0^TAx+\frac{1}{2}y_0^TAy_0+c-\frac{1}{2}b^TA^{-1}b\bigg)\bigg) \\
	= & \lim_{|x|\to\infty}\bigg(u(x)-\bigg(\frac{1}{2}x^TAx+b\cdot x+c\bigg)\bigg).
  \end{split}
\]
Hence, $u\in C^\infty(\mathbb{R}^n\setminus\Omega_0)$ is the $k$-convex solution to \eqref{eq:EDP-1}. Since $\hat{u}$ is not quasiconvex, neither is $u$. Therefore, we have proved that Theorem \ref{thm:main} and Theorem \ref{thm:main-1} can be derived from Lemma \ref{lem:AIb0}.

Now, we adopt the proof strategy from \cite{Hamel-N-S} and \cite{Wang-Xiao} to prove Lemma \ref{lem:AIb0}.
\begin{proof}[Proof of Lemma \ref{lem:AIb0}.] 
	Let $c^*$ be the positive constant given by Proposition \ref{prop:u-e} and let $u^\varepsilon$ be the solution to the problem \eqref{eq:EDP-u-e} with $c\geq c^*$. We want to show that  $u^\varepsilon$ is not quasiconvex for some $\varepsilon>0$ sufficiently small. Once this is proved, Lemma \ref{lem:AIb0} follows by choosing $\Omega_0=B_\varepsilon(x_0)$. 

	We prove it by contradiction. Assume not, then for every $\varepsilon>0$, all sublevel sets of $\tilde{u}^\varepsilon$ are convex. Let $\{\varepsilon_m\}$ be a sequence satisfying $0\notin B_{\varepsilon_m}(x_0)$ and $\varepsilon_m\to0$ as $m\to\infty$, and let $\{x_m\}$ be a sequence of points satisfying $x_m\in B_{\varepsilon_m}(x_0)\setminus\{x_0\}$.
	Since $x_0\neq0$, by Proposition \ref{prop:limit}, $\tilde{u}^{\varepsilon_m}(0)\to\psi(0)$ as $m\to\infty$. Recalling that $\tilde{u}^{\varepsilon_m}\leq\psi$ in $\mathbb{R}^n$ and noticing the choice of $x_m$, we have 
	$$\tilde{u}^{\varepsilon_m}(0)\leq\psi(0)\quad\text{and}\quad \tilde{u}^{\varepsilon_m}(x_m)=0<c=\psi(0).$$
	By virtue of the convexity of the sublevel set of $\tilde{u}^{\varepsilon_m}$, we have 
	$$\tilde{u}^{\varepsilon_m}(x)\leq\psi(0),\quad\forall\, x\in[x_m,0]$$
	where $[x_m, 0]= \{ (1-t)x_m:\, 0\leq t\leq 1\}$ denotes the closed line segment connecting $x_m$ and $0$ (the half-open segment $(\cdot, 0]$ is defined similarly).
	We will infer that 
	$$\psi(x)\leq \psi(0),\quad\forall\, x\in(x_0,0].$$ 
	Indeed, for any fixed $x\in(x_0,0]$, $x=(1-t)x_0$ for some $0<t\leq1$. Let $y_m=(1-t)x_m\in
	(x_m,0]$. Then $\tilde{u}^{\varepsilon_m}(y_m)\leq\psi(0)$.  Since $x_m\to x_0$, we have $y_m\to x$ as $m\to\infty$.
	Note that 
	$$|\tilde{u}^{\varepsilon_m}(y_m)-\psi(x)|\leq|\tilde{u}^{\varepsilon_m}(y_m)-\psi(y_m)|+|\psi(y_m)-\psi(x)|\to0$$
	as $m\to\infty$, due to the continuity of $\psi$ at $x$, alongside the locally uniform convergence of $\{\tilde{u}^{\varepsilon_m}\}$ to $\psi$ near $x \neq x_0$ established in Proposition \ref{prop:limit}.
	Therefore, $\psi(x)\leq\psi(0)$ for any $x\in(x_0,0]$.
	The continuity of $\psi$ further implies 
	$$\psi(x_0)\leq \psi(0).$$ 
	This contradicts $\psi(x_0)=\frac{1}{2}x_0^TAx_0+c>c=\psi(0)$. The proof of Lemma \ref{lem:AIb0}   is thereby complete.
\end{proof}
Note that in the proof above, we actually prove that there exists an $\varepsilon$ in any neighborhood of $0$ such that the solution $u^\varepsilon$ is quasiconvex. Together with the arbitrariness of $x_0\in B_{\frac{1}{2}}(0)\setminus\{0\}$, there exists infinitely many $\Omega_0:=B_\varepsilon(x_0)$ for which the corresponding solutions $u^\varepsilon$ are not quasiconvex, and Remark \ref{rmk} follows. 

\section{Proof of Theorem \ref{thm:main-H}}\label{sec:convexity-H}
 Our proof of Theorem \ref{thm:main-H} relies on a continuous deformation method. To carry out this deformation, it is crucial to first establish the asymptotic behavior of harmonic functions and the Gaussian curvature of their level sets at infinity. 

\begin{lemma}\label{lem:H-asy}
	Let $n\geq3$ and $\Omega$ be a smooth, strictly convex domain in $\mathbb{R}^n$ such that $B_r(\bar{x})\subset\Omega\subset B_R(\bar{x})$ for some $\bar{x} \in \Omega$ and $R > r > 0$. Suppose $u\in C^\infty(\mathbb{R}^n\setminus\Omega)$ satisfies \eqref{eq:H-EDP}.
	Then there exists a positive constant $M$ with $r^{n-2}\leq M\leq R^{n-2}$, such that as $|x|\to\infty$,
	$$u(x)=M|x|^{2-n}+O(|x|^{1-n}),$$
	$$Du(x)=-M(n-2)|x|^{-n}x+O(|x|^{-n}),$$
	$$D^2 u(x) = -M(n-2)|x|^{-n} I + Mn(n-2)|x|^{-n-2} x \otimes x + O(|x|^{-n-1})=O(|x|^{-n}).$$
\end{lemma}
\begin{proof} 
	Since $B_r(\bar{x})\subset\Omega\subset B_R(\bar{x})$, we have 
	$$\bigg(\frac{r}{|x-\bar{x}|}\bigg)^{n-2}\leq 1\leq \bigg(\frac{R}{|x-\bar{x}|}\bigg)^{n-2}\quad\text{on }\partial\Omega.$$
	The comparison principle for harmonic functions gives 
	\begin{equation}\label{eq:mM}
		\bigg(\frac{r}{|x-\bar{x}|}\bigg)^{n-2}\leq u(x)\leq \bigg(\frac{R}{|x-\bar{x}|}\bigg)^{n-2}\quad\text{in }\mathbb{R}^n\setminus\Omega.
	\end{equation}
	Multiplying by $|x|^{n-2}$ and taking the limit as $|x| \to \infty$, we deduce that the asymptotic constant $M$, if it exists, must satisfy $r^{n-2} \le M \le R^{n-2}$.
	
	To establish the existence of $M$, let $v$ be the Kelvin transform of $u$ in $\mathbb{R}^n\setminus \bar{B}_{\rho_0}(0)$ for sufficiently large $\rho_0$, that is
\[ v(x) = |x|^{2-n} u\bigg(\frac{x}{|x|^2}\bigg), \quad x\in B_{\frac{1}{\rho_0}}(0)\setminus\{0\}.\]
Since \( u \) is harmonic in $\mathbb{R}^n\setminus B_{\rho_0}(0)$, \( v \) is harmonic in the punctured ball $B_{\frac{1}{\rho_0}}(0)\setminus\{0\}$. From \eqref{eq:mM}, we know $v$ is bounded in $B_{\frac{1}{\rho_0}}(0)\setminus\{0\}$ and hence $0$ is its removable singularity. That is $v$ is harmonic and analytic in $B_{\frac{1}{\rho_0}}(0)$. Then \( v(x) = \sum_{k=0}^\infty p_k(x) \) where \( p_k \) is a homogeneous polynomial of degree \( k \). Return to \( u \),
\[ u(x) = |x|^{2-n} v\bigg(\frac{x}{|x|^2}\bigg) = \sum_{k=0}^\infty |x|^{2-n} p_k\bigg(\frac{x}{|x|^2}\bigg) = \sum_{k=0}^\infty |x|^{2(1-k)-n} p_k(x). \] 
This guarantees the existence of $M$ and $u(x) = M|x|^{2-n} + O(|x|^{1-n})$ as $|x| \to \infty$. 

Next, let $w(x)=u(x)-M|x|^{2-n}$. Then $w$ is harmonic in $\mathbb{R}^n\setminus\bar{B}_{\rho_0}(0)$ and $w(x)=O(|x|^{1-n})$ as $|x|\to\infty$. Denote $\rho=|x| $. By the interior estimate, for $|x|$ sufficiently large, 
$$|Dw(x)|\leq \frac{C_n}{\rho}\sup_{y\in B_{\frac{\rho}{2}}(x)}|w(y)|\leq C\rho^{-n},$$
$$|D^2w(x)|\leq\frac{C_n}{\rho^2}\sup_{y\in B_{\frac{\rho}{2}}(x)}|w(y)|\leq C\rho^{-n-1}.$$
Combining with the derivatives of $M|x|^{2-n}$, we obtain the desired asymptotic expansions for $Du$ and $D^2u$. 
\end{proof}

\begin{lemma}\label{lem:K-asy}
	Let $n\geq3$ and $\Omega$ be a smooth, strictly convex domain in $\mathbb{R}^n$. Suppose $u \in C^\infty(\mathbb{R}^n\setminus\Omega)$ satisfies \eqref{eq:H-EDP}. Let $K(x)$ denote the Gaussian curvature of the level set $\{y \in \mathbb{R}^n\setminus\Omega: u(y) = u(x)\}$ passing through $x$. Then, $K(x)$ satisfies the asymptotic expansion as $|x|\to\infty$,
$$K(x) = |x|^{1-n} + O(|x|^{-n}).$$  
\end{lemma}

\begin{proof} 
	As stated by Kawohl \cite{Kawohl-1988}, $u$ has no critical points, meaning that $|Du|>0$ in $\mathbb{R}^n\setminus\Omega$.
	We fix an arbitrary point $x \in \mathbb{R}^n\setminus\bar{\Omega}$ with $\rho = |x|$ sufficiently large, and choose the coordinate $\{e_1, e_2, \dots, e_n\}$ such that
	$$e_n = -\frac{D u(x)}{|D u(x)|}.$$
	It is well known that $K$ can be represented as 
	\begin{equation}\label{eq:K}
		K = (-1)^{n-1} |Du|^{-(n+1)} \sum_{\alpha,\beta} \frac{\partial \det D^2 u}{\partial u_{\alpha\beta}} u_\alpha u_\beta.
	\end{equation}
	Clearly, $u_i(x) = 0$ for $1 \le i \le n-1$. By Lemma \ref{lem:H-asy}, we have
	$$u_n(x) = -|D u(x)| = -M(n-2)\rho^{1-n} + O(\rho^{-n}),$$
	and $e_n$ is asymptotically aligned with the radial direction $\frac{x}{\rho}$, 
	$$e_n = \frac{M(n-2)\rho^{-n}x+O(\rho^{-n})}{M(n-2)\rho^{1-n} + O(\rho^{-n})} = \frac{\frac{x}{\rho} + O(\rho^{-1})}{1 + O(\rho^{-1})}= \frac{x}{\rho} + O(\rho^{-1}).$$
	Then $x \cdot e_n = \rho + O(1)$, and so
	\begin{equation}\label{eq:x-i-proj}
		x \cdot e_i = O(1) \quad \text{for } 1 \le i \le n-1,
	\end{equation}
	
	Next, we evaluate the Hessian matrix $D^2 u$. Substituting \eqref{eq:x-i-proj} into the expansion of $u_{ij} = e_i^T(D^2 u) e_j$ in Lemma \ref{lem:H-asy} yields
	$$u_{ij} = -M(n-2)\rho^{-n} \delta_{ij} + O(\rho^{-n-1}) \quad \text{for } 1 \le i, j \le n-1.$$
	Plugging these into \eqref{eq:K}, we get
	\begin{equation}\label{eq:K-sim}
		K = \frac{(-1)^{n-1}}{|D u|^{n-1}} \frac{\partial \det D^2u}{\partial u_{nn}},
	\end{equation}
	and
	\begin{equation*}
		\begin{split}
			\frac{\partial \det D^2u}{\partial u_{nn}} & = \left( -M(n-2)\rho^{-n} \right)^{n-1} + O(\rho^{-n(n-1)-1}) \\
			& = (-1)^{n-1} M^{n-1}(n-2)^{n-1}\rho^{-n^2+n} + O(\rho^{-n^2+n-1}).
		\end{split}
	\end{equation*}
	Meanwhile, 
	\begin{equation*}
		\begin{split}
			|D u|^{n-1} & = \left( M(n-2)\rho^{1-n} + O(\rho^{-n}) \right)^{n-1} \\
			& = M^{n-1}(n-2)^{n-1}\rho^{-n^2+2n-1} + O(\rho^{-n^2+2n-2}).
		\end{split}
	\end{equation*}
	Substituting these two expansions back into \eqref{eq:K-sim}, we obtain
	$$\begin{aligned} K(x) &= \frac{\rho^{-n^2+n} + O(\rho^{-n^2+n-1})}{\rho^{-n^2+2n-1} + O(\rho^{-n^2+2n-2})} \\ &= \rho^{(-n^2+n) - (-n^2+2n-1)} \left( \frac{1 + O(\rho^{-1})}{1 + O(\rho^{-1})} \right) \\ &= \rho^{1-n} \left( 1 + O(\rho^{-1}) \right) \\ &= \rho^{1-n} + O(\rho^{-n}). \end{aligned}$$
	Since $x \in \mathbb{R}^n\setminus\bar{\Omega}$ is arbitrary, this establishes the asymptotic behavior of $K$.
\end{proof}

\begin{proof}[Proof of Theorem \ref{thm:main-H}.]
	Without loss of generality, we may assume $0\in\Omega_1$. Set $\Omega_0=B_{r}(0)$ for some $r>0$ such that $B_r(0) \subset \Omega_1$ and
	$$\Omega_t=(1-t)\Omega_0+t\Omega_1,\quad\forall\, 0\leq t\leq1.$$
	Clearly, $\Omega_t$ is a smooth family of strictly convex domains.
	Let $u^t\in C^\infty(\mathbb{R}^n\setminus\Omega_t)$ be the solution to 
	\begin{equation*}
		\begin{cases}
			\Delta u^t=0  \,\qquad\text{in }\mathbb{R} ^n\setminus\bar{\Omega}_t, \\
			u^t=1 \quad\quad\quad \text{on }\partial\Omega_t, \\
			\lim\limits_{|x|\to\infty}u^t(x)=0. &
		\end{cases}
	\end{equation*}
	As stated by Kawohl \cite{Kawohl-1988}, $|Du^t|>0$ in $\mathbb{R}^n\setminus\Omega_t$. Note that $\|\partial\Omega_t\|_{C^{3,\alpha}}$ have uniform bounds, and $M(\Omega_t)$ have uniform bounds, where $M(\Omega_t)$ is the positive constant determined by $\Omega_t$ as in Lemma \ref{lem:H-asy}. By standard global Schauder estimates, we have the uniform estimates on $\|u^t\|_{C^{3,\alpha}(\mathbb{R}^n\setminus\bar{\Omega}_t)}$ with the bound depending only on the geometry of $\Omega_1$.

	Define
	$$I = \{ t \in [0,1]: K_s > 0 \text{ everywhere in } \mathbb{R}^n\setminus\bar{\Omega}_t, \forall\, 0\leq s\leq t  \},$$
	where $K_t(x)$ is the Gaussian curvature of the level set of $u^t$ passing through $x$. It is clear that $0 \in I$. Owing to the uniform estimates on $\|u^t\|_{C^{3,\alpha}(\mathbb{R}^n\setminus\bar{\Omega}_t)}$,  $I$ is relatively open in $[0,1]$. Denote $t_0=\sup I\in(0,1]$. Then $t\in I$ for any $0\leq t<t_0$.

	We argue by contradiction.  If Theorem \ref{thm:main-H} is not true, by the openness of $I$, then $0<t_0<1$ is the first time such that the  Gaussian curvature $K_{t_0}$ of the level sets of $u^{t_0}$ becomes $0$ at some point $x_{t_0}\in\mathbb{R}^n\setminus\bar{\Omega}_{t_0}$. Take a sequence $\{t_i\}$ such that $t_i\to t_0$ $(0<t_i<t_0)$. By Theorem 2.1 of \cite{Ma-Zhang-2021}, we have
	$$\psi_{t_i}=(|Du^{t_i}|^{n-3}K_{t_i})^{\frac{1}{n-1}}$$
	is superharmonic in $\mathbb{R}^n\setminus\bar{\Omega}_{t_i}$. 
	We claim that there exists a positive constant $c_0$ such that for every $i$,
	\begin{equation}\label{eq:psi-cu}
		\psi_{t_i}\geq c_0u^{t_i}\quad\text{in }\mathbb{R}^n\setminus\bar{\Omega}_{t_i}.
	\end{equation}
	Indeed, since $\|\partial\Omega_{t_i}\|_{C^{3,\alpha}}$ have uniform bounds, the strictly convex domains $\Omega_{t_i}$ satisfy a uniform exterior ball condition. That is, there exists $r_0>0$, such that for any $i$ and $y_{t_i}\in\partial\Omega_{t_i}$, we have $\bar{\Omega}_{t_i}\cap\bar{B}_{r_0}(z_{t_i})=\{y_{t_i}\}$ for some $z_{t_i}\in\mathbb{R}^n\setminus\bar{\Omega}_{t_i}$. The strong maximum principle implies $0<u^{t_i}<1$ in $\mathbb{R}^n\setminus\bar{\Omega}_{t_i}$. Assume $\Omega_1\subset B_{R_0}(0)$, then the comparison principle gives $u^{t_i}\leq (\frac{R_0}{|x|})^{n-2}$ in $\mathbb{R}^n \setminus \bar{\Omega}_{t_i}$. 
	We may assume $r_0>2R_0$, then $u^{t_i}(z_{t_i})\leq(\frac{R_0}{r+r_0})^{n-2}<\frac{1}{2}$. By the quantitative Hopf lemma (see for instance Proposition 1.34 of \cite{Han-Lin}), we obtain a uniform positive lower bound of $|Du^{t_i}|$ on $\partial\Omega_{t_i}$,
	$$|Du^{t_i}(y_{t_i})|\geq\frac{C_n}{r_0}(u^{t_i}(y_{t_i})-u^{t_i}(z_{t_i}))>\frac{C_n}{2r_0}.$$
	 Moreover, since $u^{t_i}$ is constant on $\partial\Omega_{t_i}$, $K_{t_i}$ coincides with the Gaussian curvature of $\partial\Omega_{t_i}$, which thus also admits a uniform positive lower bound due to the strict convexity of $\Omega_{t_i}$. Hence, $\psi_{t_i}$ have a uniform positive lower bound on $\partial\Omega_{t_i}$, denoted as $c_0$. On the other hand, it follows from Lemmas \ref{lem:H-asy} and \ref{lem:K-asy} that
	 $$\lim_{|x|\to\infty}\psi_{t_i}(x)=\lim_{|x|\to\infty}u^{t_i}(x)=0.$$
	Applying the comparison principle, we obtain \eqref{eq:psi-cu}, the claim is thus proved.

	The interior Schauder estimates and Arzel\`{a}-Ascoli theorem imply that $u^{t_i} \to u^{t_0}$ in $C^3$ near $x_{t_0}$.
	Sending  $i \to \infty$ in \eqref{eq:psi-cu} yields
	$$\psi_{t_0}(x_{t_0}) \ge c_0 u^{t_0}(x_{t_0}).$$
	The strong maximum principle implies $u^{t_0}(x_{t_0}) > 0$, while the degeneracy $K_{t_0}(x_{t_0}) = 0$ gives $\psi_{t_0}(x_{t_0}) = 0$, a contradiction. This completes the proof.
\end{proof}

\section*{Acknowledgments}
C. Wang is supported by Beijing Natural Science Foundation (No. 1254049), National Natural Science Foundation of China (No. 12526518) and ``the Fundamental Research Funds for the Central Universities'' in UIBE (No. 23QD04). B. Wang is supported by National Natural Science Foundation of China (No. 12271028).  Z. Wang is supported by National Natural Science Foundation of China (No. 12141105). Part of this work has been done while the second named author is visiting the Department of Mathematics ``Federigo Enriques'' of Universit\`{a} degli Studi di Milano with support from the China Scholarship Council. The hospitality of the Department is gratefully acknowledged.

\end{document}